\documentclass{amsart}
\usepackage{amssymb}
\usepackage{amsfonts}

\theoremstyle{plain}

\newtheorem{algorithm}{Algorithm}[section]

\newtheorem{theorem}[algorithm] {Theorem}

\newtheorem{corollary}[algorithm]{Corollary}

\newtheorem{exercise}[algorithm]{Exercise}
\newtheorem{lemma}[algorithm]{Lemma}

\newtheorem{proposition}[algorithm]{Proposition}
\newtheorem{remark}[algorithm]{Remark}

\input{tcilatex}

\begin{document}
\title[Principles for Deforming Nonnegative Curvature]{Some Principles for
Deforming Nonnegative Curvature}
\dedicatory{Dedicated to Karsten Grove }
\author{Peter Petersen}
\author{Frederick Wilhelm}
\maketitle
\tableofcontents

\section{Introduction}

Let $\mathcal{N}$ be the class of closed simply connected, smooth $n$%
--manifolds admitting nonnegative sectional curvature and $\mathcal{P}%
\subset \mathcal{N}$ the corresponding class for positive curvature. Known
examples suggest that $\mathcal{N}$ ought to be much larger than $\mathcal{P}
$. On the other hand, there is no known obstruction that distinguishes
between the two classes. So it is actually possible that $\mathcal{N}=%
\mathcal{P}.$

In \cite{PetWilh2} we will give a deformation of the nonnegatively curved
metric on the Gromoll-Meyer sphere \cite{GromMey} to a positively curved
metric. The purpose of this note is to elucidate a few abstract principles
that will be used in this deformation, and possibly could be helpful for
other deformations to positive curvature.

Besides a few exceptions, \cite{Cheeg}, \cite{Dear2}, \cite{GrovVerdZil}, 
\cite{GrovZil1}, \cite{GrovZil2}, \cite{Guij} all known examples of compact
nonnegatively curved manifolds are constructed as Riemannian submersions of
compact Lie groups. A result in \cite{Tapp2} then implies that the zero
curvature planes of the nonexceptional examples are contained in totally
geodesic $2$--dimensional flats. As far as we are aware, the exceptional
examples also have totally geodesic flats, provided of course that they have
any zero curvature planes at all (\cite{Dear2} and \cite{GrovVerdZil}).

All known examples with nonnegative curvature, some zero curvatures, and
positive curvature at a point, are the images of Riemannian submersions of
compact Lie groups and hence have all zero planes contained in totally
geodesic flats (\cite{EschKer}, \cite{Esch}, \cite{GromMey}, \cite{Ker1}, 
\cite{Ker2}, \cite{PetWilh1}, \cite{Tapp1}, \cite{Wilh}, and \cite{Wilk}.)
So in most cases, any attempt to put positive curvature on a known
nonnegatively curved example must confront the issue of how to put positive
curvature on a neighborhood of a totally geodesic flat torus.

More than 20 years ago Strake observed that the presence of a totally
geodesic flat torus in a nonnegatively curved manifold means that there can
be no deformation that is positive to first order. In principle, a first
order deformation should be much easier to construct and verify than a
higher order one. In fact, if $\left\{ g_{t}\right\} _{t\in \mathbb{R}}$ is $%
C^{\infty }$ family of metrics with $g_{0}$ a metric of nonnegative
curvature, and if 
\begin{equation*}
\left. \frac{\partial }{\partial t}\mathrm{sec}_{g_{t}}\,P\right\vert
_{t=0}>0
\end{equation*}%
for all planes $P$ so that $\mathrm{sec}_{g_{0}}\,P=0,$ then $g_{t}$ has
positive curvature for all sufficiently small $t>0.$

On the other hand, if for all planes $P$ with $\mathrm{sec}_{g_{0}}\,P=0$ we
have 
\begin{eqnarray*}
\left. \frac{\partial }{\partial t}\mathrm{sec}_{g_{t}}\,P\right\vert _{t=0}
&=&0\text{ and} \\
\left. \frac{\partial ^{2}}{\partial t^{2}}\mathrm{sec}_{g_{t}}\,P\right%
\vert _{t=0} &>&0,
\end{eqnarray*}%
then, without more information, we can not make any conclusion about
obtaining positive curvature. For instance, if $\Phi _{t}$ is a flow that
moves the zero planes to positive curvature, then the variation $\left( \Phi
_{t}\right) ^{\ast }g$ can satisfy the conditions above, yet clearly each of
the metrics $\left( \Phi _{t}\right) ^{\ast }g$ are isometric to $g.$

The obvious problem with such a gauge transformation is that it only moves
zero planes to new places. Unfortunately, the discussion above illustrates
that any attempt to put positive curvature on a (generic) known
nonnegatively curved example must confront this issue. It is not enough to
consider the effect of a deformation on the set, $Z,$ of zero planes of the
original metric. Instead we to have check that the curvature becomes
positive in an entire neighborhood of $Z.$

To this bleak reality we offer the following ray of hope--

\begin{quotation}
The very rigidity of totally geodesic flats can be exploited in attempts to
deform them.
\end{quotation}

The rigidity of a totally geodesic flat within a fixed nonnegatively curved
manifold is of course well known and well understood. Here we have in mind a
different sort of rigidity. We will look at certain types of deformations
that preserve totally geodesic flats, and other types of deformations that
preserve aspects of the rigidity of totally geodesic flats. The tremendous
advantage of this rigidity is that it will allow us to change one component
of the curvature tensor while controlling the change in other components.
Since the problem of prescribing the curvature tensor is highly over
determined, in general, this is an entirely unreasonable thing to expect;
nevertheless, the rigidity of totally geodesic flats will allow us to do
this in certain narrowly constrained situations.

Besides Cheeger deformations, the metric changes that we use to go from the
Gromoll-Meyer metric to our positively curved metric are

\begin{itemize}
\item a deformation that we call the \emph{Orthogonal Partial Conformal
Change}

\item scaling of the fibers of the Riemannian submersion $Sp\left( 2\right)
\longrightarrow S^{4},$ to create integrally positive curvature, and

\item another deformation that we call the \emph{Tangential Partial
Conformal Change}
\end{itemize}

To describe a general \emph{Partial Conformal Change }we start with a
distribution $\mathcal{D}\subset TM,$ and decompose our original metric as 
\begin{equation*}
g=g_{\mathcal{D}}+g_{\mathcal{D}^{\perp }}.
\end{equation*}%
We then conformally change $g_{\mathcal{D}}$ while fixing $g_{\mathcal{D}%
^{\perp }}.$

Our use of the terms \textquotedblleft \emph{Orthogonal}\textquotedblright\
and \textquotedblleft \emph{Tangential}\textquotedblright \emph{\ }is meant
to convey that our changes will be relative to distributions that are either
orthogonal or contain the original zero curvature planes respectively.

An abstraction of the orthogonal partial conformal change is discussed in
Sections 2 and 3. It preserves nonnegative curvature, the zero curvature
locus, and has the effect of redistributing certain positive curvatures
along the initial zero curvature locus. Having a broader class of
nonnegatively curved metrics could certainly be an advantage. In fact, if we
were to perform our other deformations without doing the orthogonal partial
conformal change\emph{\ }we could make the old zero planes positively
curved, but as far as we can tell would not get positive curvature. The idea
that such a change is possible goes back at least to \cite{Wals}.

The fiber scaling is the central idea of the deformation to positive
curvature on the Gromoll--Meyer sphere, $\Sigma ^{7}.$ In section 4, we
prove an abstract theorem about fiber scaling. This result implies that if
we start with the metric from \cite{Wilh} and scale the fibers of $Sp\left(
2\right) \longrightarrow S^{4},$ then we get integrally positive curvature
over the sections that have zero curvature in \cite{Wilh}. More precisely,
the zero locus in \cite{Wilh} consists of a (large) family of totally
geodesic $2$--dimensional tori. We will show that after scaling the fibers
of $Sp\left( 2\right) \longrightarrow S^{4},$ the integral of the curvature
over any of these tori becomes positive. The computation is fairly abstract,
and the argument is made in these abstract terms, so no knowledge of the
metric of \cite{Wilh} is required.

In addition to proving that fiber scaling creates integrally positive
curvature, our argument in section 4 will provide a precise formula for what
happens to the curvature of each of the old zero curvature planes. The
leading order term has both signs, so the metric with the fibers scaled has
curvatures of both signs. On the other hand, the leading order term is also
the Hessian of a function and along any \emph{one }of our originally flat
tori it can be canceled by a conformal change of metric. The details are
carried out in subsection 4.1. Thus by reading section 4 the reader can get
a quick impression of what the entire deformation does to the curvature of a
single torus that is initially totally geodesic and flat.

Unfortunately, the conformal factor required to cancel the Hessian term from
fiber scaling varies from torus to torus. Our actual deformation includes a
partial conformal change for which the distribution $\mathcal{D}$ contains
the old zero curvature planes. This is our Tangential Partial Conformal\emph{%
\ }Change. In Section 5, we describe an abstract set up for our tangential
partial conformal deformation and show that the important curvatures change
as though we had performed an actual conformal change. Combining the results
of this section with our fiber scaling and conformal change calculations
provides a method to obtain positive curvature on the initially flat planes
of the Gromoll--Meyer sphere.

Section 6 is the first place in the paper where totally geodesic flats do
not play a prominent role. Instead we detail an observation that Cheeger
deformations can be used to create positive curvature even when the initial
metric has curvatures of both signs. Modulo the so called \textquotedblleft
Cheeger Reparametrization\textquotedblright\ of the Grassmannian, Cheeger
deformations preserve positive curvatures. In addition, any plane whose
projection to the orbits \textquotedblleft corresponds\textquotedblright\ to
a positively curved plane will become positively curved provided the
deformation is carried out for a sufficiently long period.

We do not imagine that we are the first to make this observation, and in
fact, took for granted that this idea was well understood when we wrote the
first draft of \cite{PetWilh2}. We have subsequently become aware that these
ideas are not as well known as we originally assumed, so we have included
them for the sake of completeness.

The curvatures of the zero planes of \cite{Wilh} are not affected by Cheeger
deformations, but most nearby planes feel the effect. Part of the role of
long term Cheeger deformations is to simplify the problem of estimating the
curvatures in a neighborhood of the original zero curvature locus.

Sections 7 and 8 are also part of our strategy to solve this problem, and
are the sections that are most dependent on the others. While this paper is
an attempt to divide some of our deformation of the Gromoll--Meyer sphere
into digestible, abstract pieces, the reader should be aware that in at
least one respect the argument is an intertwined whole.

In Section 7, we analyze the effect of certain Cheeger deformations on our
formula for the curvatures of our tori after fiber scaling. We will show
that Cheeger deformations have the effect of compressing the bulk of these
curvatures into a small set, $T_{0}.$ Because $T_{0}$ is small the
orthogonal partial conformal change will allow us to make certain curvatures
much larger on $T_{0},$ and \textquotedblleft pay\textquotedblright\ with
only a small decrease in curvature outside of $T_{0}.$ This synergy makes
the problem of verifying positive curvature more tractable, is crucial to
our whole argument, and explained in greater detail in Section 8.

It is natural to speculate on the extent to which some (or all) of these
ideas might be useful in other deformations to positive curvature. For
example there are non simply connected examples with nonnegative curvature
that according to Synge's Theorem can not admit positive curvature, so it is
natural to ask where our methods break down in these examples. While we have
not made an exhaustive study of this question, we can point out that if a
totally geodesic flat is \emph{vertizontal} for the submersion whose fibers
are scaled, then our curvature formula shows that it will continue to be
flat. This is the case for the metrics on $\mathbb{R}P^{3}\times \mathbb{R}%
P^{2}$ and $S^{3}\times S^{2}$ in \cite{Wilk}, with respect to the isometric 
$SO\left( 3\right) $--action of that paper. Since our total argument in \cite%
{PetWilh2} is very long, there are many obstructions to using it in general.
It seems more likely that individual pieces will find other applications.%
\vspace{0.2in}

\noindent \textbf{Acknowledgement: }\emph{We are grateful to Owen Dearricott
and Burkhard Wilking for extensive conversations and suggestions about this
work, to Igor Belegradek and Burkhard Wilking for a correction to the
statement of Theorem \ref{C^2 close}, and to Igor Belegradek for other
suggestions. }

\section{Deformations Preserving Totally Geodesic Families}

In the next two sections we describe an abstract framework for our
orthogonal partial conformal change. Our exposition will be by
\textquotedblleft bootstrapping\textquotedblright , starting with some more
general metric changes.

The problem of prescribing the curvature tensor of a Riemannian $n$%
--manifold with $n\geq 4$ is highly overdetermined. In particular, it is
unreasonable to expect to change one component of the curvature tensor while
holding other components fixed. We should also not expect to change one
component of the curvature tensor while keeping the change in other
components small compared to the change in the desired component.

In the next two sections we explore exceptions to this principle that can be
traced to the rigidity of totally geodesic flat tori in nonnegatively curved
manifolds.

We begin by recalling,

\begin{exercise}
(5.4 in \cite{Pet}) Let $\gamma $ be a geodesic in $\left( M,g\right) .$ Let 
$\tilde{g}$ be another metric on $M$ which satisfies%
\begin{equation*}
g\left( \dot{\gamma},\cdot \right) =\tilde{g}\left( \dot{\gamma},\cdot
\right) :TM\longrightarrow \mathbb{R}.
\end{equation*}%
Then $\gamma $ is also a geodesic with respect to $\tilde{g}.$
\end{exercise}

A straightforward generalization is

\begin{proposition}
Let $\mathcal{S}$ be a family of totally geodesic submanifolds of $\left(
M,g\right) .$ Let $\tilde{g}$ be another metric on $M$ which satisfies%
\begin{equation*}
g\left( X,\cdot \right) =\tilde{g}\left( X,\cdot \right) :TM\longrightarrow 
\mathbb{R}
\end{equation*}%
for all vectors $X$ tangent to a totally geodesic submanifold in $\mathcal{S}%
,$ then $\mathcal{S}$ is also a family of totally geodesic submanifolds of $%
\left( M,\tilde{g}\right) .$
\end{proposition}

\begin{proof}
If $\gamma $ is any geodesic in $S\in \mathcal{S}$ with respect to $g,$ then
by the preceding exercise, $\gamma $ is a geodesic of $\left( M,\tilde{g}%
\right) .$
\end{proof}

\begin{corollary}
If the totally geodesic family $\mathcal{S}$ of the preceding proposition
consists of totally geodesic flat submanifolds for $\left( M,g\right) ,$
then it also consists of totally geodesic flat submanifolds for $\left( M,%
\tilde{g}\right) .$
\end{corollary}

\begin{proof}
The intrinsic metric on members of $\mathcal{S}$ does not change. In
particular, totally geodesic flats are preserved.
\end{proof}

Throughout the paper we set 
\begin{equation*}
\mathrm{curv}\left( X,W\right) \equiv R\left( X,W,W,X\right) .
\end{equation*}%
If 
\begin{equation*}
\mathrm{span}\left\{ X,W\right\}
\end{equation*}%
is a $0$--curvature plane of $g,$ then a nearby plane has the form 
\begin{equation*}
\Pi _{X,W,Z,V}^{\sigma ,\tau }\equiv \mathrm{span}\left\{ X+\sigma Z,W+\tau
V\right\}
\end{equation*}%
for some tangent vectors $Z$ and $V$ and some real numbers $\sigma ,\tau .$
For each choice of $\left\{ X,W,Z,V\right\} $ we then get a quartic
polynomial%
\begin{equation*}
P\left( \sigma ,\tau \right) =\mathrm{curv}\left( X+\sigma Z,W+\tau V\right)
,
\end{equation*}%
in $\sigma $ and $\tau .$ A neighborhood of the zero planes (at the zero
curvature points) can be described as 
\begin{equation*}
\left\{ \left. \Pi _{X,W,Z,V}^{\sigma ,\tau }\right\vert \mathrm{curv}\left(
X,W\right) =0,\text{ }\left( \sigma ,\tau \right) \in \left[ 0,\varepsilon %
\right] \times \left[ 0,\varepsilon \right] \right\} .
\end{equation*}

Assuming that $M$ is compact and $\varepsilon $ is sufficiently small, we
can arrange this representation so that all of the polynomials $P\left(
\sigma ,\tau \right) $ are positive on $\left[ 0,\varepsilon \right] \times %
\left[ 0,\varepsilon \right] ,$except at $\left( \sigma ,\tau \right)
=\left( 0,0\right) .$

It is much easier to deform the metric within nonnegative curvature if, in
addition, the total quadratic term of $P\left( \sigma ,\tau \right) $
satisfies the following nondegeneracy condition

\begin{equation*}
\sigma ^{2}\mathrm{curv}\left( Z,W\right) +2\sigma \tau \left( R\left(
X,W,V,Z\right) +R\left( X,V,W,Z\right) \right) +\tau ^{2}\mathrm{curv}\left(
X,V\right) >0\text{ for all }\left( \sigma ,\tau \right) \in S^{1}.
\end{equation*}%
We call this the \emph{Quadratic Nondegeneracy Condition.}

\begin{theorem}
\label{C^2 close} Suppose that $\left( M,g\right) $ is compact and
nonnegatively curved and all of its zero planes are contained in a family $%
\mathcal{S}$ of totally geodesic flat submanifolds, and satisfy the
quadratic nondegeneracy condition.

Let $\tilde{g}$ be obtained from $g$ as in the preceding proposition.

Then $\left( M,\tilde{g}\right) $ is nonnegatively curved along the union of
the family $\mathcal{S}$ with precisely the same $0$ curvature planes as $g,$
provided $\tilde{g}$ is sufficiently close to $g$ in the $C^{2}$--topology.
\end{theorem}

\begin{remark}
This result can be viewed as an abstraction of Theorem 2.1 in \cite{Wals}.
\end{remark}

\begin{proof}
Since $g$ and $\tilde{g}$ are $C^{2}$--close any $0$--curvature planes of $%
\tilde{g}$ must be close in the Grassmannian to $0$--curvature planes of $g.$

We must show that 
\begin{equation*}
\tilde{P}\left( \sigma ,\tau \right) =\mathrm{curv}^{\tilde{g}}\left(
X+\sigma Z,W+\tau V\right)
\end{equation*}%
is also nonnegative on $\left[ 0,\varepsilon \right] \times \left[
0,\varepsilon \right] $ and that it only vanishes when $\sigma =\tau =0.$

Because $X$ and $W$ are tangent to a $0$--curvature plane in a nonnegatively
curved manifold%
\begin{equation*}
R^{g}\left( X,W\right) W=R^{g}\left( W,X\right) X=0.
\end{equation*}%
Since they are also tangent to a totally geodesic flat that is preserved
under our deformation we have 
\begin{equation*}
R^{\tilde{g}}\left( X,W\right) W=R^{\tilde{g}}\left( W,X\right) X=0.
\end{equation*}%
So the constant and linear terms of $P\left( \sigma ,\tau \right) $ and $%
\tilde{P}\left( \sigma ,\tau \right) $ vanish.

Combining the quadratic nondegeneracy condition with the fact that 
\begin{equation*}
P\left( \sigma ,\tau \right) \equiv \mathrm{curv}^{g}\left( X+\sigma
Z,W+\tau V\right) \geq 0,
\end{equation*}%
and only vanishes within $\left[ 0,\varepsilon \right] \times \left[
0,\varepsilon \right] $ when $\sigma =\tau =0,$ we conclude that $\tilde{P}%
\left( \sigma ,\tau \right) $ is nonnegative and only vanishes within $\left[
0,\varepsilon \right] \times \left[ 0,\varepsilon \right] $ when $\sigma
=\tau =0,$ provided the coefficients of $P$ and $\tilde{P}$ are sufficiently
close.

Thus $\left( M,\tilde{g}\right) $ is nonnegatively curved on the union of
the members of $\mathcal{S}$ if $\tilde{g}$ is sufficiently close to $g$ in
the $C^{2}$--topology.
\end{proof}

A problem with this theorem is that it does not tell us about the curvature
of points in $\left( M,\tilde{g}\right) $ that are not at a $0$--curvature
point of $\left( M,g\right) $. Of course there are various reasons why we
might or might not know about these curvatures. In \cite{PetWilh2}, we will
apply the following idea.

\begin{corollary}
\label{S with a submersion}Let $\left( M,g\right) $ be nonnegatively curved.
Suppose $\pi :\left( M,g\right) \longrightarrow \Sigma $ is a Riemannian
submersion. Suppose also that the lifts of $0$--planes of $\Sigma $ are
tangent to a family $\mathcal{S}$ of totally geodesic flat submanifolds of $%
M,$ and that the image%
\begin{equation*}
\cup _{S\in \mathcal{S}}\pi \left( S\right)
\end{equation*}%
contains a neighborhood $U$ of all of the points of $\Sigma $ that have $0$%
--curvatures. Suppose that the quadratic nondegeneracy condition is
satisfied on horizontal planes.

Let $\tilde{g}$ be $C^{2}$--close to $g$ and satisfy 
\begin{equation*}
g\left( X,\cdot \right) =\tilde{g}\left( X,\cdot \right)
\end{equation*}%
for all vectors tangent to a totally geodesic submanifold in $\mathcal{S}.$

If $\pi :\left( M,\tilde{g}\right) \longrightarrow \Sigma $ is a Riemannian
submersion, then the metric induced on $\Sigma $ is nonnegatively curved
with precisely the same $0$ curvature planes as $g.$
\end{corollary}

\section{Orthogonal Partial Conformal Change}

With a few qualifications, the Orthogonal Partial Conformal Change in \cite%
{PetWilh2} fits into the basic set up of the preceding corollary for the
submersion $Sp\left( 2\right) \longrightarrow \Sigma ^{7}$. The main
deficiency is that the deformation of \cite{PetWilh2} is only $C^{1}$--small.

Although the preceding corollary is false for arbitrary $C^{1}$--small
deformations, there is a narrowly constrained situation where it holds.

The main tool is proven using Cartan formalism (\cite{Spiv}, Chap 7).

\begin{lemma}
\label{C^1 principle}Suppose that $\left\{ E_{i}\right\} $ is an orthonormal
frame for $g$ with dual coframe $\left\{ \theta ^{i}\right\} .$ Suppose that 
$\tilde{\theta}^{i}=\phi ^{i}\theta ^{i}$ is an orthonormal coframe for $%
\tilde{g},$ where $\phi ^{i}$ are smooth functions on $M.$ Assume that%
\begin{equation*}
d\phi ^{i}=\psi ^{i}\theta ^{1}
\end{equation*}%
and that 
\begin{equation*}
d\psi ^{i}=\lambda ^{i}\theta ^{1}
\end{equation*}%
for some other smooth functions $\psi ^{i}$ and $\lambda ^{i}.$ If the
functions $\phi ^{i}$ are close to $1$ in the $C^{1}$--topology, then the
only components of $R^{\tilde{g}}\left( \tilde{E}_{i},\tilde{E}_{j},\tilde{E}%
_{k},\tilde{E}_{l}\right) $ that are not close to $R^{g}$ are the terms that
up to symmetries of the curvature tensor can be reduced to $R^{\tilde{g}%
}\left( \tilde{E}_{1},\tilde{E}_{i},\tilde{E}_{i},\tilde{E}_{1}\right) .$
\end{lemma}

\begin{remark}
Note that the meaning of \textquotedblleft close\textquotedblright\ depends
on $g.$
\end{remark}

\begin{proof}
Following (\cite{Spiv}, Chap 7) we define $\left\{ b_{jk}^{i}\right\}
,\left\{ a_{jk}^{i}\right\} ,$and $\left\{ \omega _{j}^{i}\right\} $ by 
\begin{eqnarray*}
d\theta ^{i} &=&\frac{1}{2}\sum_{j,k=1}^{n}b_{jk}^{i}\theta ^{j}\wedge
\theta ^{k}, \\
a_{jk}^{i} &=&\frac{1}{2}\left( b_{jk}^{i}+b_{ki}^{j}-b_{ij}^{k}\right)
\end{eqnarray*}%
\begin{equation*}
\omega _{j}^{i}=\sum_{k=1}^{n}a_{jk}^{i}\theta ^{k}.
\end{equation*}%
It then follows (\cite{Spiv}, Chap 7) that%
\begin{eqnarray*}
b_{jk}^{i} &=&-b_{kj}^{i}\text{ and} \\
a_{jk}^{i} &=&-a_{ik}^{j}
\end{eqnarray*}%
The forms $\Omega _{j}^{i}$ 
\begin{equation*}
\Omega _{j}^{i}\equiv d\omega _{j}^{i}+\sum_{k=1}^{n}\omega _{k}^{i}\wedge
\omega _{j}^{k}
\end{equation*}%
are then curvatures. Specifically%
\begin{equation*}
g\left( R\left( X,Y\right) E_{j},E_{i}\right) =\Omega _{j}^{i}\left(
X,Y\right) .
\end{equation*}

We now check how these functions get changed for the new frame.%
\begin{eqnarray*}
d\tilde{\theta}^{i} &=&d\left( \phi ^{i}\theta ^{i}\right) \\
&=&d\phi ^{i}\wedge \theta ^{i}+\frac{1}{2}\sum_{j,k=1}^{n}\phi
^{i}b_{jk}^{i}\theta ^{j}\wedge \theta ^{k} \\
&=&\psi ^{i}\theta ^{1}\wedge \theta ^{i}+\frac{1}{2}\sum_{j,k=1}^{n}\phi
^{i}b_{jk}^{i}\theta ^{j}\wedge \theta ^{k} \\
&=&\frac{1}{2}\frac{\psi ^{i}}{\phi ^{1}\phi ^{i}}\tilde{\theta}^{1}\wedge 
\tilde{\theta}^{i}-\frac{1}{2}\frac{\psi ^{i}}{\phi ^{1}\phi ^{i}}\tilde{%
\theta}^{i}\wedge \tilde{\theta}^{1} \\
&&+\frac{1}{2}\sum_{j,k=1}^{n}\frac{\phi ^{i}}{\phi ^{j}\phi ^{k}}b_{jk}^{i}%
\tilde{\theta}^{j}\wedge \tilde{\theta}^{k}.
\end{eqnarray*}

So the only $\tilde{b}_{jk}^{i}$s that depend on $\psi $ are 
\begin{equation*}
\tilde{b}_{1i}^{i}=-\tilde{b}_{i1}^{i}=\frac{\psi ^{i}}{\phi ^{1}\phi ^{i}}+%
\frac{\phi ^{i}}{\phi ^{1}\phi ^{i}}b_{1i}^{i}.
\end{equation*}%
So among the 
\begin{equation*}
\tilde{a}_{jk}^{i}=\frac{1}{2}\left( \tilde{b}_{jk}^{i}+\tilde{b}_{ki}^{j}-%
\tilde{b}_{ij}^{k}\right)
\end{equation*}%
the ones potentially affected by $\psi $ are $\tilde{a}_{1i}^{i},\tilde{a}%
_{i,1}^{i},$ and $\tilde{a}_{i,i}^{1}.$However, the antisymmetry $\tilde{a}%
_{jk}^{i}=-\tilde{a}_{ik}^{j}$ implies that $\tilde{a}_{i,1}^{i}=0,$ and
that $\tilde{a}_{1i}^{i}=-\tilde{a}_{i,i}^{1}.$The antisymmetry of the $b$s
then gives us 
\begin{equation*}
\tilde{a}_{1i}^{i}=-\tilde{a}_{i,i}^{1}=\tilde{b}_{1i}^{i}.
\end{equation*}

So in fact, the $\tilde{a}$s that depend on $\psi $ are 
\begin{eqnarray*}
\tilde{a}_{1i}^{i} &=&-\tilde{a}_{i,i}^{1} \\
&=&\frac{1}{2}\left( \tilde{b}_{1i}^{i}+\tilde{b}_{ii}^{1}-\tilde{b}%
_{i1}^{i}\right) \\
&=&\left( \frac{\psi ^{i}}{\phi ^{1}\phi ^{i}}+\frac{\phi ^{i}}{\phi
^{1}\phi ^{i}}b_{1i}^{i}\right) \\
&=&\left( \frac{\psi ^{i}}{\phi ^{1}\phi ^{i}}+\frac{\phi ^{i}}{\phi
^{1}\phi ^{i}}a_{1i}^{i}\right)
\end{eqnarray*}

Thus the only $\omega _{j}^{i}$s that depend on $\psi $ are%
\begin{eqnarray*}
\tilde{\omega}_{1}^{i} &=&-\tilde{\omega}_{i}^{1} \\
&=&\tilde{a}_{1i}^{i}\tilde{\theta}^{i}+\sum_{k\neq i}\tilde{a}_{1k}^{i}%
\tilde{\theta}^{k} \\
&=&\frac{\psi ^{i}}{\phi ^{1}\phi ^{i}}\tilde{\theta}^{i}+\sum_{k}a_{1k}^{i}%
\theta ^{k}+O\left( C^{0}\right)
\end{eqnarray*}%
where by $O\left( C^{0}\right) $ we mean $O\left( \max \left\{ 1-\phi
^{i}\right\} \right) \sum_{k}\theta ^{k}.$

It follows that the only $\Omega _{j}^{i}$s that depend on the $\lambda $s
are 
\begin{equation*}
\Omega _{1}^{i}=-\Omega _{i}^{1}=d\tilde{\omega}_{1}^{i}+\sum_{k=1}^{n}%
\omega _{k}^{i}\wedge \omega _{1}^{k}.
\end{equation*}

Only the first term feels this \textquotedblleft $C^{2}$--effect%
\textquotedblright . It is%
\begin{equation*}
d\tilde{\omega}_{1}^{i}=d\left( \frac{\psi ^{i}}{\phi ^{1}\phi ^{i}}\tilde{%
\theta}^{i}\right) +d\omega _{1}^{i}+O\left( C^{1}\right) ,
\end{equation*}

where by $O\left( C^{1}\right) $ we mean 
\begin{equation*}
O\left( \max \left\{ 1-\phi ^{i},\psi ^{i}\right\} \right) \sum_{k}d\theta
^{k}+d\left( O\left( \max \left\{ 1-\phi ^{i}\right\} \right) \sum_{k}\theta
^{k}\right) .
\end{equation*}

We conclude that 
\begin{equation*}
d\tilde{\omega}_{1}^{i}=\lambda ^{i}\theta ^{1}\wedge \tilde{\theta}%
^{i}+d\omega _{1}^{i}+O\left( C^{1}\right) .
\end{equation*}%
So we note that the only curvatures affected by the $C^{2}$ change are the
sectional curvature spanned by $E_{1}$ and $E_{i}$
\end{proof}

In our applications we will also need to know something more specific about
how the other components of the curvature tensor change with such a
deformation.

\begin{corollary}
\label{Owen}If at most one of the indices $\left\{ i,j,k,l\right\} $ is $1,$
then%
\begin{equation*}
\left\vert R^{\tilde{g}}\left( \tilde{E}_{i},\tilde{E}_{j},\tilde{E}_{k},%
\tilde{E}_{l}\right) -R\left( E_{i},E_{j},E_{k},E_{l}\right) \right\vert
\leq O\left( \max \left\{ \psi ^{i},1-\phi ^{i}\right\} \right) O\left(
\max_{i,j,k}\left\{ \left\vert \omega _{j}^{i}\right\vert ,\left\vert
da_{jk}^{i}\right\vert \right\} \right)
\end{equation*}
\end{corollary}

\begin{proof}
We have%
\begin{equation*}
\left\vert \sum_{p=1}^{n}\tilde{\omega}_{p}^{i}\wedge \tilde{\omega}%
_{j}^{p}\left( \tilde{E}_{l},\tilde{E}_{k}\right) -\sum_{p=1}^{n}\omega
_{p}^{i}\wedge \omega _{j}^{p}\left( E_{l},E_{k}\right) \right\vert \leq
\left( O\left( \max \left\{ 1-\phi ^{i},\psi ^{i}\right\} \max \left\{
b_{jk}^{i}\right\} \right) \right) ^{2},
\end{equation*}%
and%
\begin{eqnarray*}
d\tilde{\omega}_{j}^{i}\left( \tilde{E}_{l},\tilde{E}_{p}\right) &=&d\left[
\sum_{k=1}^{n}\tilde{a}_{jk}^{i}\tilde{\theta}^{k}\right] \left( \tilde{E}%
_{l},\tilde{E}_{p}\right) \\
&=&\left[ \sum_{k=1}^{n}d\tilde{a}_{jk}^{i}\wedge \tilde{\theta}^{k}+\tilde{a%
}_{jk}^{i}d\tilde{\theta}^{k}\right] \left( \tilde{E}_{l},\tilde{E}%
_{p}\right) \\
&=&d\omega _{j}^{i}\left( E_{l},E_{p}\right) \\
&&+O\left( \max \left\{ \psi ^{i},1-\phi ^{i}\right\} \right)
\sum_{k=1}^{n}da_{jk}^{i}\wedge \theta ^{k}\left( E_{l},E_{p}\right)
+O\left( \max \left\{ \psi ^{i},1-\phi ^{i}\right\} \right)
\sum_{k=1}^{n}a_{jk}^{i}d\theta ^{k}\left( E_{l},E_{p}\right)
\end{eqnarray*}

Combining these formulas yields%
\begin{equation*}
\left\vert R^{\tilde{g}}-R\right\vert \leq O\left( C^{1}\right) O\left(
\max_{i,j,k}\left\{ \left\vert \omega _{j}^{i}\right\vert ,\left\vert
b_{jk}^{i}\right\vert ,\left\vert da_{jk}^{i}\right\vert \right\} \right)
\end{equation*}%
Since $O\left( \max_{i,j,k}\left\{ \left\vert b_{jk}^{i}\right\vert \right\}
\right) \leq O\left( \max_{i,j}\left\vert \omega _{j}^{i}\right\vert \right) 
$ the result follows.
\end{proof}

In addition to the general set up of Corollary \ref{S with a submersion} we
also assume,

\begin{description}
\item[1] There is a smooth distance function $r$ defined on on a
neighborhood of $\mathcal{S}$ whose gradient we call $X.$

\item[2] $X$ is tangent to all the flats in $\mathcal{S}.$
\end{description}

Let $\mathcal{O}$ be a distribution that is normal to each $S\in \mathcal{S}%
, $ and let 
\begin{equation*}
\varphi \equiv f\circ r
\end{equation*}%
where $f:\mathbb{R}\longrightarrow \mathbb{R},$ and is constant outside of a
compact interval.

Let $\tilde{g}$ be obtained from $g$ by multiplying the lengths of all
vectors in $\mathcal{O}$ by $\varphi ,$ while keeping the orthogonal
complement and the metric on the orthogonal complement of $\mathcal{O}$
fixed. That is, $\tilde{g}$ is obtained from $g$ by doing a partial
conformal change with distribution $\mathcal{D}=\mathcal{O}$ and conformal
factor $\varphi ^{2}.$

Now let $\left\{ E_{i}\right\} $ be an orthonormal frame for $g$ with $%
X=E_{1},$ and $\mathrm{span}\left\{ E_{2},\ldots ,E_{p}\right\} =\mathcal{O}%
. $ Setting $\phi ^{i}=\varphi $ for $i=2,\ldots p,$ and $\phi ^{i}\equiv 1$
otherwise, then gives an example of the above lemma.

Applying the last formula in the proof of the preceding lemma to our
situation yields.

\begin{proposition}
\label{redistr curv}For $V\in \mathcal{O}$%
\begin{equation*}
R^{\tilde{g}}\left( V,X,X,V\right) =R^{g}\left( V,X,X,V\right) -\varphi
^{\prime \prime }\left\vert V\right\vert _{g}^{2}\left\vert X\right\vert
_{g}^{2}+O\left( C^{1}\right)
\end{equation*}%
where $\varphi ^{\prime \prime }=D_{X}D_{X}\left( \varphi \right) =f^{\prime
\prime }\circ r.$
\end{proposition}

Combining this with the previous lemma and the proofs of Theorem \ref{C^2
close} and Corollary \ref{S with a submersion} gives us.

\begin{theorem}
\label{C^1 close copy(1)}There is an $\varepsilon >0$ so that $\left( \Sigma
,\tilde{g}\right) $ is nonnegatively curved provided,

\begin{itemize}
\item $\varphi $ is sufficiently close to $1$ in the $C^{1}$--topology, and

\item For any $V\in \mathcal{O}$ 
\begin{equation*}
\varepsilon \left\langle R^{g}\left( X,V\right) V,X\right\rangle -\varphi
^{\prime \prime }\left\vert V\right\vert _{g}^{2}\left\vert X\right\vert
_{g}^{2}>0
\end{equation*}%
Moreover $\tilde{g}$ has precisely the same $0$ curvature planes as $g.$
\end{itemize}
\end{theorem}

\begin{remark}
There are three further complications when the ideas of the previous two
sections are applied in \cite{PetWilh2}.

First, the parameter that describes the orthogonal partial conformal change
is related to the parameter that describes one of the Cheeger deformations,
and hence we can not merely use Lemma \ref{C^1 principle}, but also must use
Corollary \ref{Owen}. This is not a difficult point, but it is better to
address it concretely.

Second, the actual field, $X,$ that is used has a singularity along a set
where it is multi-valued. This causes some of the connection forms for the
original metric to blow up near these points. This means that the $C^{1}$%
--change in $\varphi $ can actually have a large effect on some curvatures.
The set up is such that this effect only increases curvatures. This is also
not a difficult point, but it is one that is best addressed concretely.

Lastly, we must verify that the quadratic nondegeneracy condition of Theorem %
\ref{C^2 close} holds on $\Sigma ^{7}.$ In the end we shall see that it only
holds generically. This problem and its resolution turn out to be related to
the second\ problem about the blow up of certain connection forms, as the
places where the nondegeneracy condition fails are precisely at the poles of 
$X.$

In subsection 6.1, we provide a tool that shows how Cheeger deformations can
be useful in simplifying the problem of verifying the nondegeneracy
condition. We apply this tool in \cite{PetWilh2} to show that $\Sigma ^{7}$
satisfies the nondegeneracy condition except at the poles of $X.$

This situation is far from ideal since it means that we almost have
quadratic \emph{degeneracy}\textbf{\ }as we approach a pole of $X.$ We will
show that the orthogonal partial conformal change actually improves this
situation, and makes it as nice as it can be.

This is accomplished by proving a further result that shows that the
orthogonal partial conformal change gives us quadratic nondegeneracy
condition, in a quantitative sense, as close as we please to the poles of $%
X. $ This result exploits the blow up of the connection forms near the poles
of $X,$and is also proven in \cite{PetWilh2}.
\end{remark}

\section{Integrally Positive Curvature}

Here we give abstract criteria that are sufficient to create integrally
positive curvature on a totally geodesic flat, when the fibers of a
Riemannian submersion are scaled. The only application of the theorem that
we are aware of is to the Gromoll-Meyer sphere with the metric from \cite%
{Wilh}. The issue of why the metric of \cite{Wilh} satisfies the hypotheses
of this theorem will be addressed in \cite{PetWilh2}.

Throughout this section let $\left( M,g_{0}\right) $ be a Riemannian
manifold with nonnegative sectional curvature and 
\begin{equation*}
\pi :\left( M,g_{0}\right) \longrightarrow B
\end{equation*}%
a Riemannian submersion. Let $g_{s}$ be the metric obtained from $g_{0}$ by
scaling the lengths of the fibers of $\pi $ by 
\begin{equation*}
\sqrt{1-s^{2}}.
\end{equation*}%
As usual we use the superscripts $^{\mathcal{H}}$ and $^{\mathcal{V}}$ to
denote the horizontal and vertical parts of the vectors, $R$ and $A$ are the
curvature and $A$-tensors for the unperturbed metric $g,$ $R^{g_{s}}$
denotes the new curvature tensor of $g_{s},$ and $R^{B}$ is the curvature
tensor of the base. We use the term \textquotedblleft geodesic
field\textquotedblright\ for any field $X$ so that $\nabla _{X}X=0.$

\begin{theorem}
\label{pos int}Let $T\subset M$ be a totally geodesic, flat torus spanned by
commuting, orthogonal, geodesic fields $X$ and $W$ such that $X$ is
horizontal for $\pi $ and $D\pi \left( W\right) $ is a Jacobi field along
the integral curves of $D\pi \left( X\right) .$

Then%
\begin{equation}
R^{g_{s}}\left( X,W,W,X\right) =-\frac{s^{2}}{2}\left( D_{X}\left(
D_{X}\left\vert W^{\mathcal{H}}\right\vert ^{2}\right) \right)
+s^{4}\left\vert A_{X}W^{\mathcal{V}}\right\vert ^{2}.
\end{equation}%
In particular, if $c$ is an integral curve of $D\pi \left( X\right) $ from a
zero of $\left\vert W^{\mathcal{H}}\right\vert $ to a maximum of $\left\vert
W^{\mathcal{H}}\right\vert $ along $c,$ then 
\begin{equation*}
\int_{c}\mathrm{curv}_{g_{s}}\left( X,W\right) =s^{4}\int_{c}\left\vert
A_{X}W^{\mathcal{V}}\right\vert ^{2}.
\end{equation*}%
So the curvature of \textrm{span}$\left\{ X,W\right\} $ is integrally
positive along $c,$ provided $\left\vert A_{X}W^{\mathcal{V}}\right\vert
^{2} $ is not identically $0$ along $c.$
\end{theorem}

The reader should note that the above curvature formula is as important as
the fact that the integral is positive. Since $X$ is a geodesic field, the
larger term $-\frac{s^{2}}{2}\left( D_{X}\left( D_{X}\left\vert W^{\mathcal{H%
}}\right\vert ^{2}\right) \right) $ is the Hessian, $\mathrm{Hess}f\left(
X,X\right) ,$ of the function 
\begin{equation*}
f=-\frac{s^{2}}{2}\left\vert W^{\mathcal{H}}\right\vert ^{2}.
\end{equation*}%
Therefore, we can cancel it with a conformal change involving $f$. Such a
conformal change will create other terms of order $s^{4}$ in our expression
for $\mathrm{curv}_{g_{s}}\left( X,W\right) $. To compare these terms with $%
s^{4}\left\vert A_{X}W^{\mathcal{V}}\right\vert ^{2},$ we will evaluate $%
A_{X}W^{\mathcal{V}}$ in the presence of some additional hypotheses, after
we prove the theorem above. These additional hypotheses will also allow us
to obtain formulas for the $\left( 1,3\right) $--tensor, $R^{g_{s}}\left(
W,X\right) X$ and the horizontal part of the $\left( 1,3\right) $--tensor, $%
R^{g_{s}}\left( X,W\right) W.$ To actually put positive curvature on the
Gromoll--Meyer sphere, or indeed to perturb a neighborhood of any totally
geodesic flat to positive curvature, these formulas will of course be
necessary.

After refining our formula for $\mathrm{curv}_{g_{s}}\left( X,W\right) ,$ we
will explain in the next subsection precisely how to combine fiber scaling
and a conformal change to put positive curvature on a single initially flat
torus, subject to a few additional hypotheses.

Scaling the fibers of a Riemannian submersion was dubbed the
\textquotedblleft canonical variation\textquotedblright\ in \cite{Bes}. One
can find formulas for how curvature changes under the canonical variation in
any of \cite{Bes}, \cite{Dear1}, \cite{GromDur}, or \cite{GromWals}. To
ultimately get positive curvature on the Gromoll-Meyer sphere, we have to
control the curvature tensor in an entire neighborhood in the Grassmannian,
so we will need several of these formulas. In fact, since the particular
\textquotedblleft $W$\textquotedblright\ that we have in mind is neither
horizontal nor vertical for $\pi ,$ we need multiple formulas just to find 
\textrm{curv}$\left( X,W\right) .$

Given vertical vectors $U,V\in \mathcal{V}$ and horizontal vectors $X,Y,Z\in 
\mathcal{H},$ for $\pi :M\rightarrow B$ we have\addtocounter{algorithm}{1} 
\begin{eqnarray}
\left( R^{g_{s}}\left( X,V\right) U\right) ^{\mathcal{H}} &=&\left(
1-s^{2}\right) \left( R\left( X,V\right) U\right) ^{\mathcal{H}}+\left(
1-s^{2}\right) s^{2}A_{A_{X}U}V  \notag \\
R^{g_{s}}(V,X)Y &=&\left( 1-s^{2}\right) R(V,X)Y+s^{2}\left( R(V,X)Y\right)
^{\mathcal{V}}+s^{2}A_{X}A_{Y}V  \notag \\
R^{g_{s}}\left( X,Y\right) Z &=&\left( 1-s^{2}\right) R\left( X,Y\right)
Z+s^{2}\left( R\left( X,Y\right) Z\right) ^{\mathcal{V}}+s^{2}R^{B}\left(
X,Y\right) Z  \label{Detlef equation}
\end{eqnarray}

To eventually understand the curvature in a neighborhood of the
Gromoll-Meyer $0$-locus, we will need formulas for 
\begin{equation*}
R^{g_{s}}\left( W,X\right) X\text{ and }\left( R^{g_{s}}\left( X,W\right)
W\right) ^{\mathcal{H}}
\end{equation*}%
where $X$ is as above and $W$ is an arbitrary vector in $TM.$

Splitting $W$ into horizontal and vertical parts and applying the formulas
above we obtain the following.

\begin{lemma}
\label{Detlef}Let $X$ be a horizontal vector for $\pi $ and let $W$ be an
arbitrary vector in $TM.$ Then%
\begin{eqnarray*}
R^{g_{s}}\left( W,X\right) X &=&\left( 1-s^{2}\right) R(W,X)X+s^{2}\left(
R(W,X)X\right) ^{\mathcal{V}} \\
&&+s^{2}R^{B}\left( W^{\mathcal{H}},X\right) X+s^{2}A_{X}A_{X}W^{\mathcal{V}}
\\
\left( R^{g_{s}}\left( X,W\right) W\right) ^{\mathcal{H}} &=&\left(
1-s^{2}\right) \left( R\left( X,W\right) W\right) ^{\mathcal{H}} \\
&&+\left( 1-s^{2}\right) s^{2}A_{A_{X}W^{\mathcal{V}}}W^{\mathcal{V}%
}+s^{2}R^{B}\left( X,W^{\mathcal{H}}\right) W^{\mathcal{H}}
\end{eqnarray*}
\end{lemma}

\begin{remark}
Notice that the first curvature terms vanish in both formulas on the totally
geodesic flat tori.
\end{remark}

Using the fact that $\mathrm{curv}^{g_{0}}\left( X,W\right) =0$ and either
of the formulas for $R^{g_{s}}\left( W,X\right) X$ or $\left(
R^{g_{s}}\left( X,W\right) W\right) ^{\mathcal{H}}$ we have 
\begin{equation}
\mathrm{curv}^{g_{s}}\left( X,W\right) =s^{2}\mathrm{curv}^{B}\left( X,W^{%
\mathcal{H}}\right) -\left( 1-s^{2}\right) s^{2}\left\vert A_{X}W^{\mathcal{V%
}}\right\vert ^{2}.
\end{equation}%
Since $D\pi \left( W^{\mathcal{H}}\right) $ is a Jacobi field along $c,$ and
writing $W^{\mathcal{H}}$ for $D\pi \left( W^{\mathcal{H}}\right) $ we have 
\begin{eqnarray*}
\mathrm{curv}^{B}\left( X,W^{\mathcal{H}}\right) &=&-\left\langle \nabla
_{X}^{B}\nabla _{X}^{B}W^{\mathcal{H}},W^{\mathcal{H}}\right\rangle \\
&=&-D_{X}\left\langle \nabla _{X}^{B}W^{\mathcal{H}},W^{\mathcal{H}%
}\right\rangle +\left\langle \nabla _{X}^{B}W^{\mathcal{H}},\nabla
_{X}^{B}W^{\mathcal{H}}\right\rangle \\
&=&-\frac{1}{2}D_{X}D_{X}\left\langle W^{\mathcal{H}},W^{\mathcal{H}%
}\right\rangle +\left\langle \nabla _{X}^{B}W^{\mathcal{H}},\nabla
_{X}^{B}W^{\mathcal{H}}\right\rangle
\end{eqnarray*}%
Since $\nabla _{X}W\equiv 0$ we have 
\begin{eqnarray*}
0 &\equiv &\nabla _{X}W \\
&=&\nabla _{X}W^{\mathcal{H}}+\nabla _{X}W^{\mathcal{V}}.
\end{eqnarray*}%
The horizontal part of this equation gives us%
\begin{equation*}
A_{X}W^{\mathcal{V}}=-\left( \nabla _{X}W^{\mathcal{H}}\right) ^{\mathcal{H}%
}.
\end{equation*}%
Identifying$\ \left( \nabla _{X}W^{\mathcal{H}}\right) ^{\mathcal{H}}$ with $%
\nabla _{X}^{B}W^{\mathcal{H}}$ and substituting into the formula for $%
\mathrm{curv}^{B}\left( X,W^{\mathcal{H}}\right) $ we obtain%
\begin{equation*}
\mathrm{curv}^{B}\left( X,W^{\mathcal{H}}\right) =-\frac{1}{2}%
D_{X}D_{X}\left\vert W^{\mathcal{H}}\right\vert ^{2}+\left\vert A_{X}W^{%
\mathcal{V}}\right\vert ^{2}
\end{equation*}%
Substituting this into our formula for $\mathrm{curv}^{g_{s}}\left(
X,W\right) $ yields 
\begin{eqnarray*}
\mathrm{curv}^{g_{s}}\left( X,W\right) &=&-s^{2}\frac{1}{2}%
D_{X}D_{X}\left\vert W^{\mathcal{H}}\right\vert ^{2}+s^{2}\left\vert A_{X}W^{%
\mathcal{V}}\right\vert ^{2}-\left( 1-s^{2}\right) s^{2}\left\vert A_{X}W^{%
\mathcal{V}}\right\vert ^{2} \\
&=&-s^{2}\frac{1}{2}D_{X}D_{X}\left\vert W^{\mathcal{H}}\right\vert
^{2}+s^{4}\left\vert A_{X}W^{\mathcal{V}}\right\vert ^{2},
\end{eqnarray*}%
proving Theorem \ref{pos int}.

To help evaluate $A_{X}W^{\mathcal{V}}$, we add some general assumptions
about the Riemannian submersion $\pi :\left( M,g_{0}\right) \rightarrow B.$

\begin{itemize}
\item There is an isometric action by $G$ on $M$ that is by symmetries of $%
\pi .$

\item The intrinsic metrics on the principal orbits of $G$ in $B$ are
homotheties of each other.

\item The normal distribution to the orbits of $G$ on $B$ is integrable.
\end{itemize}

In addition we add some specific conditions to the hypotheses of Theorem \ref%
{pos int}:

\begin{itemize}
\item $W^{\mathcal{H}}$ is a Killing field for the $G$--action on $B.$

\item $D\pi \left( X\right) $ is invariant under the action that $G$ induces
on $B.$

\item $D\pi \left( X\right) $ is orthogonal to the orbits of $G.$
\end{itemize}

Since the normal distribution to the orbits of $G$ on $B$ is integrable we
can extend any normal vector $Z$ to a $G$--invariant normal field $Z$.
Writing $X$ for $D\pi \left( X\right) $ it then follows that all terms in
the Koszul formula for 
\begin{equation*}
\left\langle \nabla _{W^{\mathcal{H}}}^{B}X,Z\right\rangle
\end{equation*}%
vanish. In particular, $\nabla _{W^{\mathcal{H}}}^{B}X$ is tangent to the
orbits of $G.$

If $K$ is another Killing field for $G,$ then $X$ commutes with $K$ as well
as $W^{\mathcal{H}},$ thus $\left[ K,W^{\mathcal{H}}\right] $ is
perpendicular to $X$ as it is again a Killing field. Combining this with our
hypothesis that the intrinsic metrics on the principal orbits of $G$ in $B$
are homotheties of each other, we see from Koszul's formula that $\nabla
_{W^{\mathcal{H}}}X$ is proportional to $W^{\mathcal{H}}$ and can be
calculated by%
\begin{eqnarray*}
\left\langle \nabla _{W^{\mathcal{H}}}X,W^{\mathcal{H}}\right\rangle
&=&\left\langle \nabla _{X}W^{\mathcal{H}},W^{\mathcal{H}}\right\rangle \\
&=&\frac{1}{2}D_{X}\left\vert W^{\mathcal{H}}\right\vert ^{2} \\
&=&\left\vert W^{\mathcal{H}}\right\vert D_{X}\left\vert W^{\mathcal{H}%
}\right\vert ,\text{ so} \\
\nabla _{W^{\mathcal{H}}}X &=&\frac{D_{X}\left\vert W^{\mathcal{H}%
}\right\vert }{\left\vert W^{\mathcal{H}}\right\vert }W^{\mathcal{H}}.
\end{eqnarray*}%
Since 
\begin{equation*}
A_{X}W^{\mathcal{V}}=-\left( \nabla _{X}W^{\mathcal{H}}\right) ^{\mathcal{H}}
\end{equation*}%
we conclude

\begin{lemma}
\label{Abstract A--tensor} With the additional hypotheses mentioned above%
\begin{equation*}
A_{X}W^{\mathcal{V}}=-\frac{D_{X}\left\vert W^{\mathcal{H}}\right\vert }{%
\left\vert W^{\mathcal{H}}\right\vert }W^{\mathcal{H}}.
\end{equation*}
\end{lemma}

Plugging this into our curvature formula we get 
\begin{equation}
\mathrm{curv}^{g_{s}}\left( X,W\right) =-s^{2}\frac{1}{2}D_{X}D_{X}\left%
\vert W^{\mathcal{H}}\right\vert ^{2}+s^{4}\left\vert D_{X}\left\vert W^{%
\mathcal{H}}\right\vert \right\vert ^{2}.  \label{curv after s}
\end{equation}%
As we've mentioned, to get positive curvature on the Gromoll-Meyer sphere,
we will have to understand certain other components of the $\left(
1,3\right) $ curvature tensor.

\begin{lemma}
\label{R^B ( H, X) X} Using $W^{\mathcal{H}}$ for $d\pi \left( W\right) $
and $X$ for $d\pi \left( X\right) $%
\begin{equation*}
R^{B}\left( W^{\mathcal{H}},X\right) X=-\left( \frac{D_{X}D_{X}\left\vert W^{%
\mathcal{H}}\right\vert }{\left\vert W^{\mathcal{H}}\right\vert }\right) W^{%
\mathcal{H}}
\end{equation*}
\end{lemma}

\begin{proof}
Since $X$ is a geodesic field and $W^{\mathcal{H}}$ is a Jacobi field along
the integral curves of $X$%
\begin{equation*}
R^{B}\left( W^{\mathcal{H}},X\right) X=-\nabla _{X}\nabla _{X}W^{\mathcal{H}%
}.
\end{equation*}%
We discovered above that 
\begin{equation*}
\nabla _{X}W^{\mathcal{H}}=\nabla _{W^{\mathcal{H}}}X=\frac{D_{X}\left\vert
W^{\mathcal{H}}\right\vert }{\left\vert W^{\mathcal{H}}\right\vert }W^{%
\mathcal{H}}.
\end{equation*}%
Thus 
\begin{eqnarray*}
R^{B}\left( W^{\mathcal{H}},X\right) X &=&-\nabla _{X}\left( \frac{%
D_{X}\left\vert W^{\mathcal{H}}\right\vert }{\left\vert W^{\mathcal{H}%
}\right\vert }W^{\mathcal{H}}\right) \\
&=&-D_{X}\left( \frac{D_{X}\left\vert W^{\mathcal{H}}\right\vert }{%
\left\vert W^{\mathcal{H}}\right\vert }\right) W^{\mathcal{H}}-\left( \frac{%
D_{X}\left\vert W^{\mathcal{H}}\right\vert }{\left\vert W^{\mathcal{H}%
}\right\vert }\nabla _{X}W^{\mathcal{H}}\right) \\
&=&-\left( \frac{\left\vert W^{\mathcal{H}}\right\vert D_{X}D_{X}\left\vert
W^{\mathcal{H}}\right\vert -\left( D_{X}\left\vert W^{\mathcal{H}%
}\right\vert \right) ^{2}}{\left\vert W^{\mathcal{H}}\right\vert ^{2}}%
\right) W^{\mathcal{H}}-\left( \frac{D_{X}\left\vert W^{\mathcal{H}%
}\right\vert }{\left\vert W^{\mathcal{H}}\right\vert }\right) ^{2}W^{%
\mathcal{H}} \\
&=&-\left( \frac{D_{X}D_{X}\left\vert W^{\mathcal{H}}\right\vert }{%
\left\vert W^{\mathcal{H}}\right\vert }\right) W^{\mathcal{H}}.
\end{eqnarray*}
\end{proof}

\begin{lemma}
\label{R^B (X, H)H} Using $W^{\mathcal{H}}$ for $d\pi \left( W\right) $ and $%
X$ for $d\pi \left( X\right) $%
\begin{equation*}
R^{B}\left( X,W^{\mathcal{H}}\right) W^{\mathcal{H}}=-\left\vert W^{\mathcal{%
H}}\right\vert \nabla _{X}\left( \mathrm{grad}\left\vert W^{\mathcal{H}%
}\right\vert \right) .
\end{equation*}
\end{lemma}

\begin{proof}
Let $Z$ be any vector field. Using that $W^{\mathcal{H}}$ is a Killing field
we get 
\begin{eqnarray*}
\left\langle \nabla _{W^{\mathcal{H}}}W^{\mathcal{H}},Z\right\rangle
&=&-\left\langle \nabla _{Z}W^{\mathcal{H}},W^{\mathcal{H}}\right\rangle \\
&=&-\frac{1}{2}D_{Z}\left\langle W^{\mathcal{H}},W^{\mathcal{H}}\right\rangle
\\
&=&-\frac{1}{2}D_{Z}\left\vert W^{\mathcal{H}}\right\vert ^{2} \\
&=&-\left\vert W^{\mathcal{H}}\right\vert D_{Z}\left\vert W^{\mathcal{H}%
}\right\vert \\
&=&-\left\langle \left\vert W^{\mathcal{H}}\right\vert \mathrm{grad}%
\left\vert W^{\mathcal{H}}\right\vert ,Z\right\rangle
\end{eqnarray*}%
showing that 
\begin{equation*}
\nabla _{W^{\mathcal{H}}}W^{\mathcal{H}}=-\left\vert W^{\mathcal{H}%
}\right\vert \mathrm{grad}\left\vert W^{\mathcal{H}}\right\vert .
\end{equation*}%
Thus%
\begin{eqnarray*}
R^{B}\left( X,W^{\mathcal{H}}\right) W^{\mathcal{H}} &=&\nabla _{X}\nabla
_{W^{\mathcal{H}}}W^{\mathcal{H}}-\nabla _{W^{\mathcal{H}}}\nabla _{X}W^{%
\mathcal{H}} \\
&=&-\nabla _{X}\left( \left\vert W^{\mathcal{H}}\right\vert \mathrm{grad}%
\left\vert W^{\mathcal{H}}\right\vert \right) -\nabla _{W^{\mathcal{H}%
}}\left( \frac{D_{X}\left\vert W^{\mathcal{H}}\right\vert }{\left\vert W^{%
\mathcal{H}}\right\vert }W^{\mathcal{H}}\right) \\
&=&-\left( D_{X}\left\vert W^{\mathcal{H}}\right\vert \right) \mathrm{grad}%
\left\vert W^{\mathcal{H}}\right\vert -\left( \left\vert W^{\mathcal{H}%
}\right\vert \nabla _{X}\mathrm{grad}\left\vert W^{\mathcal{H}}\right\vert
\right) -\frac{D_{X}\left\vert W^{\mathcal{H}}\right\vert }{\left\vert W^{%
\mathcal{H}}\right\vert }\nabla _{W^{\mathcal{H}}}W^{\mathcal{H}} \\
&=&-\left( D_{X}\left\vert W^{\mathcal{H}}\right\vert \right) \mathrm{grad}%
\left\vert W^{\mathcal{H}}\right\vert -\left( \left\vert W^{\mathcal{H}%
}\right\vert \nabla _{X}\mathrm{grad}\left\vert W^{\mathcal{H}}\right\vert
\right) +\frac{D_{X}\left\vert W^{\mathcal{H}}\right\vert }{\left\vert W^{%
\mathcal{H}}\right\vert }\left\vert W^{\mathcal{H}}\right\vert \mathrm{grad}%
\left\vert W^{\mathcal{H}}\right\vert \\
&=&-\left( \left\vert W^{\mathcal{H}}\right\vert \nabla _{X}\mathrm{grad}%
\left\vert W^{\mathcal{H}}\right\vert \right)
\end{eqnarray*}
\end{proof}

Combining the calculations above we have

\begin{lemma}
\label{Abstract (1,3) tensors} Let $X$ and $W$ be as in Theorem \ref{pos int}%
. Then 
\begin{eqnarray*}
R^{g_{s}}\left( W,X\right) X &=&-s^{2}\left( \frac{D_{X}D_{X}\left\vert W^{%
\mathcal{H}}\right\vert }{\left\vert W^{\mathcal{H}}\right\vert }\right) W^{%
\mathcal{H}}-s^{2}\frac{D_{X}\left\vert W^{\mathcal{H}}\right\vert }{%
\left\vert W^{\mathcal{H}}\right\vert }A_{X}W^{\mathcal{H}} \\
\left( R^{g_{s}}\left( X,W\right) W\right) ^{\mathcal{H}} &=&-\left(
1-s^{2}\right) s^{2}\frac{D_{X}\left\vert W^{\mathcal{H}}\right\vert }{%
\left\vert W^{\mathcal{H}}\right\vert }A_{W^{\mathcal{H}}}W^{\mathcal{V}%
}-s^{2}\left\vert W^{\mathcal{H}}\right\vert \nabla _{X}\left( \mathrm{grad}%
\left\vert W^{\mathcal{H}}\right\vert \right) .
\end{eqnarray*}
\end{lemma}

\begin{remark}
The two $A$--tensors $A_{X}W^{\mathcal{H}}$ and $A_{W^{\mathcal{H}}}W^{%
\mathcal{V}}$ involve derivatives of vectors that are not tangent or normal
to the totally geodesic tori. They cannot be determined abstractly, and are
in fact dependent on the particular geometry. We give estimates for them in
the case of the Gromoll-Meyer sphere in the section of \cite{PetWilh2}
called \textquotedblleft Concrete $A$--tensor estimates\textquotedblright .
\end{remark}

\subsection{Positive Curvature on a Single Initially Flat Torus}

In this subsection we will explain how our fiber scaling calculations can be
combined with a conformal change to put positive curvature on a \emph{single 
}flat torus, $T,$ that satisfies the hypotheses of the previous section as
well as a few other mild hypotheses. This fact may seem reassuring, however,
we emphasize that for the following reasons it is not sufficient to get
positive curvature on the Gromoll--Meyer sphere.

\begin{itemize}
\item We will not learn (much) about the curvatures of nearby planes,

\item The Gromoll--Meyer sphere with the metric of \cite{Wilh} has many
totally geodesic flat tori. For reasons that we shall make explicit in the
next section, fiber scaling combined with a conformal change can not be used
to put positive curvature on all of these tori simultaneously.
\end{itemize}

In Section 4 we discuss an abstract situation that allows for a certain type
of partial conformal change to affect certain curvatures in the same way as
an actual conformal change. By combining the results of that section and
this one we will have a method that puts positive curvature on all of the
totally geodesic flats of the Gromoll--Meyer sphere simultaneously, modulo
the question of verifying that the Gromoll--Meyer sphere satisfies all of
the necessary hypotheses. This last question is resolved in \cite{PetWilh2},
as well as the issue of actually verifying positive curvature.

Imagine that $T\subset M$ is a totally geodesic flat torus spanned by
geodesic fields $X$ and $W$ satisfying all of the hypotheses of the previous
section. Let 
\begin{equation*}
\tilde{T}:\left[ 0,\pi \right] \times \left[ 0,l\right] \longrightarrow M
\end{equation*}%
be a parameterization of $T$ with $X$ a unit field whose integral curves, $%
c_{s_{0}},$ are 
\begin{equation*}
c_{s_{0}}:t\longmapsto \tilde{T}\left( t,s_{0}\right) .
\end{equation*}%
In particular, the integral curves of $X$ are periodic with minimal period $%
\pi .$

We also assume that along each $c_{s_{0}}$ the key function $\left\vert W^{%
\mathcal{H}}\right\vert $

\begin{itemize}
\item is periodic in the first variable with period $\frac{\pi }{2},$ i.e. $%
\left\vert W^{\mathcal{H}}\right\vert _{\tilde{T}\left( t,s_{0}\right)
}=\left\vert W^{\mathcal{H}}\right\vert _{\tilde{T}\left( t+\frac{\pi }{2}%
,s_{0}\right) },$

\item has zeros only when $t$ is $0,\frac{\pi }{2},$ $\pi ,\ldots $ ,

\item and maxima only when $t$ is $\frac{\pi }{4},\frac{3\pi }{4},\frac{5\pi 
}{4},\ldots $.
\end{itemize}

We also assume that

\begin{equation*}
\mathrm{dist}\left( \tilde{T}\left( \left\{ 0\right\} \times \left[ 0,l%
\right] \right) ,\cdot \right)
\end{equation*}%
is smooth on $\tilde{T}\left( \left( 0,\frac{\pi }{4}\right) \times \left[
0,l\right] \right) $ with gradient $X.$

To simplify notation we set 
\begin{equation*}
\psi =\left\vert W^{\mathcal{H}}\right\vert .
\end{equation*}%
So after scaling the fibers of $\pi $ by $\sqrt{1-s^{2}}$ we have from \ref%
{curv after s} 
\begin{equation}
\mathrm{curv}_{g_{s}}\left( X,W\right) =-s^{2}\left( D_{X}\left( \psi
D_{X}\psi \right) \right) +s^{4}\left( D_{X}\psi \right) ^{2}.
\label{fiber scale psi tilde}
\end{equation}

We remind the reader that after the conformal change $\tilde{g}=e^{2f}g_{s}$
we will have 
\begin{eqnarray*}
e^{-2f}\mathrm{curv}_{\tilde{g}}\left( X,W\right) &=&\mathrm{curv}%
_{g_{s}}\left( X,W\right) -\left\vert W\right\vert _{g_{s}}^{2}\mathrm{Hess}%
f\left( X,X\right) -\mathrm{Hess}f\left( W,W\right) \\
&&+\left( D_{X}f\right) ^{2}\left\vert W\right\vert _{g_{s}}^{2}-\left\vert 
\mathrm{grad}f\right\vert ^{2}\left\vert W\right\vert _{g_{s}}^{2},
\end{eqnarray*}%
provided $X$ is unit and $W$ is perpendicular to $\mathrm{grad}f$ (cf \cite%
{Pet} Exercise 3.5)

Our choice of conformal factor will look like 
\begin{equation*}
f=-\frac{s^{2}}{2\left( 1-s^{2}\right) }\frac{\psi ^{2}}{\left\vert
W\right\vert ^{2}}+\text{\textrm{\ }a much smaller term.}
\end{equation*}%
The first conformal term $-\left\vert W\right\vert _{g_{s}}^{2}\mathrm{Hess}%
f\left( X,X\right) $ will nearly cancel with the leading term $-s^{2}\left(
D_{X}\left( \psi D_{X}\psi \right) \right) $ in $\mathrm{curv}_{g_{s}}\left(
X,W\right) .$ For our initial metric $\nabla _{W}W\equiv 0,$ so $\mathrm{Hess%
}f\left( W,W\right) $ has order $s^{4},$ as do the other two conformal
terms, $\left( D_{X}f\right) ^{2}\left\vert W\right\vert _{g_{s}}^{2}$ and $%
\left\vert \nabla f\right\vert ^{2}\left\vert W\right\vert _{g_{s}}^{2}.$ In
the remainder of this section we will see more precisely what these terms
actually are.

To do this we name the \textquotedblleft much smaller term\textquotedblright
, $E.$ The function $E$ has the form 
\begin{equation*}
E=s^{4}I\circ \mathrm{dist}\left( \tilde{T}\left( \left\{ 0\right\} \times %
\left[ 0,l\right] \right) ,\cdot \right)
\end{equation*}%
where $I:\mathbb{R}\longrightarrow \mathbb{R}$ is a function that satisfies 
\begin{equation*}
I^{\prime }\left( 0\right) =I^{\prime }\left( \frac{\pi }{4}\right) =0,
\end{equation*}

Thus 
\begin{equation*}
\mathrm{grad}\text{ }f=-\frac{s^{2}}{\left( 1-s^{2}\right) \left\vert
W\right\vert ^{2}}\psi \mathrm{grad}\psi +s^{4}I^{\prime }X
\end{equation*}

To understand the effect that this conformal change has on our curvatures we
will need to know the Hessian of $f$, and hence a covariant derivative that
we have yet to compute.

\begin{proposition}
\label{Nabla_W W}%
\begin{equation*}
\nabla _{W}^{g_{s}}W=-s^{2}\psi \mathrm{grad}\psi ,
\end{equation*}
\end{proposition}

\begin{proof}
Before the fiber scaling $\nabla _{W}W=0.$ Breaking $W$ into horizontal and
vertical parts and using the Koszul formula we get 
\begin{eqnarray*}
\nabla _{W}^{g_{s}}W &=&\nabla _{W^{\mathcal{V}}}^{g_{s}}W^{\mathcal{V}%
}+\nabla _{W^{\mathcal{V}}}^{g_{s}}W^{\mathcal{H}}+\nabla _{W^{\mathcal{H}%
}}^{g_{s}}W^{\mathcal{V}}+\nabla _{W^{\mathcal{H}}}^{g_{s}}W^{\mathcal{H}} \\
&=&\left( \nabla _{W^{\mathcal{V}}}W^{\mathcal{V}}\right) ^{\mathcal{V}%
}+\left( 1-s^{2}\right) \left( \nabla _{W^{\mathcal{V}}}W^{\mathcal{V}%
}\right) ^{\mathcal{H}} \\
&&+\left( \nabla _{W^{\mathcal{V}}}W^{\mathcal{H}}\right) ^{\mathcal{V}%
}+\left( 1-s^{2}\right) \left( \nabla _{W^{\mathcal{V}}}W^{\mathcal{H}%
}\right) ^{\mathcal{H}} \\
&&+\left( \nabla _{W^{\mathcal{H}}}W^{\mathcal{V}}\right) ^{\mathcal{V}%
}+\left( 1-s^{2}\right) \left( \nabla _{W^{\mathcal{H}}}W^{\mathcal{V}%
}\right) ^{\mathcal{H}} \\
&&+\nabla _{W^{\mathcal{H}}}W^{\mathcal{H}}
\end{eqnarray*}%
Rearranging terms and using the fact that $\nabla _{W}W=0$ yields 
\begin{equation*}
\nabla _{W}^{g_{s}}W=-s^{2}\left[ \left( \nabla _{W^{\mathcal{V}}}W^{%
\mathcal{V}}\right) +\nabla _{W^{\mathcal{V}}}W^{\mathcal{H}}+\nabla _{W^{%
\mathcal{H}}}W^{\mathcal{V}}\right] ^{\mathcal{H}}.
\end{equation*}%
We also have $\left( \nabla _{W}W\right) ^{\mathcal{H}}=0,$ so 
\begin{equation*}
\left[ \nabla _{W^{\mathcal{V}}}W^{\mathcal{V}}+\nabla _{W^{\mathcal{V}}}W^{%
\mathcal{H}}+\nabla _{W^{\mathcal{H}}}W^{\mathcal{V}}+\nabla _{W^{\mathcal{H}%
}}W^{\mathcal{H}}\right] ^{\mathcal{H}}=0.
\end{equation*}

Thus 
\begin{eqnarray*}
\nabla _{W}^{g_{s}}W &=&-s^{2}\left[ \left( \nabla _{W^{\mathcal{V}}}W^{%
\mathcal{V}}\right) +\nabla _{W^{\mathcal{V}}}W^{\mathcal{H}}+\nabla _{W^{%
\mathcal{H}}}W^{\mathcal{V}}\right] ^{\mathcal{H}} \\
&=&s^{2}\nabla _{W^{\mathcal{H}}}W^{\mathcal{H}} \\
&=&-s^{2}\left\vert W^{\mathcal{H}}\right\vert \mathrm{grad}\left\vert W^{%
\mathcal{H}}\right\vert \\
&=&-s^{2}\psi \mathrm{grad}\psi ,
\end{eqnarray*}%
where we have used an equation in the proof of Lemma \ref{R^B (X, H)H} for
the next to last inequality.
\end{proof}

\begin{proposition}
\label{Hessians copy(1)}%
\begin{equation*}
\mathrm{Hess}f\left( X,X\right) =-\frac{s^{2}}{\left( 1-s^{2}\right)
\left\vert W\right\vert ^{2}}D_{X}\left( \psi D_{X}\psi \right)
+s^{4}I^{\prime \prime }
\end{equation*}%
\begin{equation*}
\mathrm{Hess}f\left( W,W\right) =-\frac{s^{4}}{\left( 1-s^{2}\right)
\left\vert W\right\vert ^{2}}\psi ^{2}\left\vert \mathrm{grad\,}\psi
\right\vert ^{2}+O\left( s^{6}\right)
\end{equation*}
\end{proposition}

\begin{proof}
Since%
\begin{equation*}
\mathrm{grad}f=-\frac{s^{2}}{\left( 1-s^{2}\right) \left\vert W\right\vert
^{2}}\psi \mathrm{grad}\psi +s^{4}I^{\prime }X
\end{equation*}

we have 
\begin{eqnarray*}
\mathrm{Hess}f\left( X,X\right) &=&-\frac{s^{2}}{\left( 1-s^{2}\right)
\left\vert W\right\vert ^{2}}\left\langle \nabla _{X}\left( \psi \mathrm{grad%
}\psi \right) ,X\right\rangle +s^{4}\left\langle \nabla _{X}\left( I^{\prime
}X\right) ,X\right\rangle \\
&=&-\frac{s^{2}}{\left( 1-s^{2}\right) \left\vert W\right\vert ^{2}}\left(
\left( D_{X}\psi \right) ^{2}+\psi \left\langle \nabla _{X}\left( \mathrm{%
grad}\psi \right) ,X\right\rangle \right) +s^{4}I^{\prime \prime } \\
&=&-\frac{s^{2}}{\left( 1-s^{2}\right) \left\vert W\right\vert ^{2}}\left(
\left( D_{X}\psi \right) ^{2}+\psi D_{X}D_{X}\psi \right) +s^{4}I^{\prime
\prime } \\
&=&-\frac{s^{2}}{\left( 1-s^{2}\right) \left\vert W\right\vert ^{2}}%
D_{X}\left( \psi D_{\zeta }\psi \right) +s^{4}I^{\prime \prime }
\end{eqnarray*}

Since $W$ is perpendicular to $\mathrm{grad}f$ we have 
\begin{eqnarray*}
\mathrm{Hess}f\left( W,W\right) &=&\left\langle \nabla _{W}\mathrm{grad}%
f,W\right\rangle \\
&=&-\left\langle \mathrm{grad}f,\nabla _{W}W\right\rangle
\end{eqnarray*}%
Using the previous proposition this gives us 
\begin{eqnarray*}
\mathrm{Hess}f\left( W,W\right) &=&-\left\langle -\frac{s^{2}}{\left(
1-s^{2}\right) \left\vert W\right\vert ^{2}}\psi \mathrm{grad}\psi
,-s^{2}\psi \mathrm{grad}\psi \right\rangle -s^{4}\left\langle I^{\prime
}X,-s^{2}\psi \mathrm{grad}\psi \right\rangle \\
&=&-\frac{s^{4}}{\left( 1-s^{2}\right) \left\vert W\right\vert ^{2}}\psi
^{2}\left\vert \mathrm{grad}\psi \right\vert ^{2}+O\left( s^{6}\right)
\end{eqnarray*}
\end{proof}

\begin{proposition}
\label{curv(X,W)}After fiber scaling and the conformal change we have 
\begin{equation*}
e^{-2f}\mathrm{curv}\left( X,W\right) =s^{4}\left( D_{X}\psi \right)
^{2}+s^{4}\frac{\psi ^{2}}{\left\vert W\right\vert ^{2}}\left( D_{X}\psi
\right) ^{2}+s^{4}\frac{\psi ^{2}}{\left\vert W\right\vert ^{2}}D_{X}\left(
\psi D_{X}\psi \right) -s^{4}I^{\prime \prime }\left\vert W\right\vert
^{2}+O\left( s^{6}\right)
\end{equation*}
\end{proposition}

\begin{remark}
We pick $I^{\prime \prime }$ so that the first four terms are $O\left(
s^{4}\right) $. The first two are positive except at $t=\frac{\pi }{4}.$ The
third can have either sign, and since the integral of $I^{\prime \prime }$
over an integral curve of $X$ is $0,$ the term $s^{4}I^{\prime \prime
}\left\vert W\right\vert ^{2}$ also has both signs. After proving the
proposition we will argue that the integral $e^{-2f}\mathrm{curv}\left(
X,W\right) $ is positive, and hence that an appropriate choice of $I^{\prime
\prime }$ will give us point wise positive curvature.
\end{remark}

\begin{proof}
Combining $\left\vert X\right\vert \equiv 1,$ equation \ref{fiber scale psi
tilde}, the formula for the curvature of a conformal change (\cite{Pet},
exercise 3.5), and the fact that $W$ is perpendicular to $\mathrm{grad}f$ we
have 
\begin{eqnarray*}
e^{-2f}\mathrm{curv}_{\tilde{g}}\left( X,W\right) &=&-s^{2}\left(
D_{X}\left( \psi D_{X}\psi \right) \right) +s^{4}\left( D_{X}\psi \right)
^{2}-\left\vert W\right\vert _{g_{s}}^{2}\mathrm{Hess}f\left( X,X\right) -%
\mathrm{Hess}f\left( W,W\right) \\
&&+\left( D_{X}f\right) ^{2}\left\vert W\right\vert _{g_{s}}^{2}-\left\vert 
\mathrm{grad}f\right\vert ^{2}\left\vert W\right\vert _{g_{s}}^{2}.
\end{eqnarray*}%
To evaluate this we will need

\begin{eqnarray*}
\left\vert W\right\vert _{g_{s}}^{2} &=&\left( 1-s^{2}\right) \left\vert W^{%
\mathcal{V}}\right\vert ^{2}+\left\vert W^{\mathcal{H}}\right\vert ^{2} \\
&=&\left\vert W\right\vert ^{2}-s^{2}\left\vert \left( W\right) ^{\mathcal{V}%
}\right\vert ^{2} \\
&=&\left\vert W\right\vert ^{2}-s^{2}\left( \left\vert W\right\vert
^{2}-\left\vert W^{\mathcal{H}}\right\vert ^{2}\right) \\
&=&\left( 1-s^{2}\right) \left\vert W\right\vert ^{2}+s^{2}\left\vert W^{%
\mathcal{H}}\right\vert ^{2} \\
&=&\left( 1-s^{2}\right) \left\vert W\right\vert ^{2}+s^{2}\psi ^{2}.
\end{eqnarray*}%
Combining this with the previous proposition we see that the sum of the
first and third term is 
\begin{eqnarray*}
&&-s^{2}\left( D_{X}\left( \psi D_{X}\psi \right) \right) -\left\vert
W\right\vert _{s}^{2}\mathrm{Hess}f\left( X,X\right) \\
&=&-s^{2}\left( D_{X}\left( \psi D_{X}\psi \right) \right) +\left(
1-s^{2}\right) \left\vert W\right\vert ^{2}\frac{s^{2}}{\left(
1-s^{2}\right) \left\vert W\right\vert ^{2}}D_{X}\left( \psi D_{X}\psi
\right) \\
&&+s^{2}\psi ^{2}\frac{s^{2}}{\left( 1-s^{2}\right) \left\vert W\right\vert
^{2}}D_{X}\left( \psi D_{X}\psi \right) -s^{4}I^{\prime \prime }\left\vert
W\right\vert _{g_{s}}^{2} \\
&=&s^{4}\frac{\psi ^{2}}{\left\vert W\right\vert ^{2}}D_{X}\left( \psi
D_{X}\psi \right) -s^{4}I^{\prime \prime }\left\vert W\right\vert
^{2}+O\left( s^{6}\right) .
\end{eqnarray*}%
The sum of the fourth and last terms is 
\begin{eqnarray*}
-\mathrm{Hess}f\left( W,W\right) -\left\vert \mathrm{grad}f\right\vert
^{2}\left\vert W\right\vert _{g_{s}}^{2} &=&\frac{s^{4}}{\left(
1-s^{2}\right) \left\vert W\right\vert ^{2}}\psi ^{2}\left\vert \mathrm{grad}%
\psi \right\vert ^{2}-\frac{s^{4}}{\left( 1-s^{2}\right) ^{2}\left\vert
W\right\vert ^{4}}\left\vert \psi \mathrm{grad}\psi \right\vert
^{2}\left\vert W\right\vert _{s}^{2}+O\left( s^{6}\right) \\
&=&O\left( s^{6}\right) .
\end{eqnarray*}%
The fifth term of our curvature formula is 
\begin{equation*}
\left( D_{X}f\right) ^{2}\left\vert W\right\vert _{g_{s}}^{2}=s^{4}\frac{%
\psi ^{2}}{\left\vert W\right\vert ^{2}}\left( D_{X}\psi \right)
^{2}+O\left( s^{6}\right) .
\end{equation*}

Combining equations we obtain 
\begin{equation*}
e^{-2f}\mathrm{curv}\left( X,W\right) =s^{4}\left( D_{X}\psi \right)
^{2}+s^{4}\frac{\psi ^{2}}{\left\vert W\right\vert ^{2}}\left( D_{X}\psi
\right) ^{2}+s^{4}\frac{\psi ^{2}}{\left\vert W\right\vert ^{2}}D_{X}\left(
\psi D_{X}\psi \right) -s^{4}I^{\prime \prime }\left\vert W\right\vert
^{2}+O\left( s^{6}\right)
\end{equation*}

as desired.
\end{proof}

To understand the sign of the above formula we will need to understand some
relationships between the integrals of the first three terms.

\begin{proposition}
\label{integrals of psi tilde}Let $\gamma :\left[ 0,\frac{\pi }{4}\right]
\longrightarrow M$ be an integral curve of $X$ with $\gamma \left( 0\right)
\in \tilde{T}\left( \left\{ 0\right\} \times \left[ 0,l\right] \right) .$
Then using $\psi ^{\prime }$ for $D_{X}\psi $%
\begin{equation*}
\int_{\gamma }\psi ^{2}\left( \psi ^{\prime }\right) ^{2}dt=-\frac{1}{3}%
\int_{\gamma }\psi ^{2}\left( \psi \psi ^{\prime \prime }\right) dt
\end{equation*}%
\begin{equation*}
\int_{\gamma }\psi ^{2}\left( \psi \psi ^{\prime }\right) ^{\prime
}dt=-2\int_{\gamma }\psi ^{2}\left( \psi ^{\prime }\right) ^{2}dt
\end{equation*}
\end{proposition}

\begin{proof}
The first equation follows from integration by parts 
\begin{eqnarray*}
\int_{\gamma }\psi ^{2}\left( \psi ^{\prime }\right) ^{2}dt &=&\int_{\gamma
}\psi ^{\prime }\left( \psi ^{2}\psi ^{\prime }\right) dt \\
&=&\left. \psi ^{\prime }\frac{1}{3}\psi ^{3}\right\vert _{0}^{\frac{\pi }{4}%
}-\int_{\gamma }\psi ^{\prime \prime }\frac{1}{3}\psi ^{3}dt \\
&=&-\frac{1}{3}\int_{\gamma }\psi ^{\prime \prime }\psi ^{3}dt
\end{eqnarray*}

So%
\begin{eqnarray*}
\int_{\gamma }\psi ^{2}\left( \psi \psi ^{\prime }\right) ^{\prime }dt
&=&\int_{\gamma }\psi ^{2}\left\{ \left( \psi ^{\prime }\right) ^{2}+\psi
\psi ^{\prime \prime }\right\} dt \\
&=&\int_{\gamma }\psi ^{2}\left\{ \left( \psi ^{\prime }\right) ^{2}-3\left(
\psi ^{\prime }\right) ^{2}\right\} dt \\
&=&-2\int_{\gamma }\psi ^{2}\left( \psi ^{\prime }\right) ^{2}dt
\end{eqnarray*}
\end{proof}

Using the second equation of the previous proposition we can re-write the
integral of our curvature over $\gamma $ as 
\begin{eqnarray*}
\int_{\gamma }e^{-2f}\mathrm{curv}\left( X,W\right) &=&\int_{\gamma
}s^{4}\left( D_{X}\psi \right) ^{2}+s^{4}\frac{\psi ^{2}}{\left\vert
W\right\vert ^{2}}\left( D_{X}\psi \right) ^{2}+s^{4}\frac{\psi ^{2}}{%
\left\vert W\right\vert ^{2}}D_{X}\left( \psi D_{X}\psi \right)
-s^{4}I^{\prime \prime }\left\vert W\right\vert ^{2}+O\left( s^{6}\right) \\
&=&\int_{\gamma }s^{4}\left( D_{X}\psi \right) ^{2}-s^{4}\frac{\psi ^{2}}{%
\left\vert W\right\vert ^{2}}\left( D_{X}\psi \right) ^{2}-s^{4}I^{\prime
\prime }\left\vert W\right\vert ^{2}+O\left( s^{6}\right) .
\end{eqnarray*}%
Since $\psi ^{2}=\left\vert W^{\mathcal{H}}\right\vert ^{2}$ we always have 
\begin{equation*}
\frac{\psi ^{2}}{\left\vert W\right\vert ^{2}}\leq 1.
\end{equation*}%
Since we also have $\psi ^{2}\left( 0\right) =0,$ the inequality is strict
at least for a while. It follows that the integral 
\begin{equation*}
\int_{\gamma }s^{4}\left( D_{X}\psi \right) ^{2}-s^{4}\frac{\psi ^{2}}{%
\left\vert W\right\vert ^{2}}\left( D_{X}\psi \right) ^{2}\geq O\left(
s^{4}\right) >0
\end{equation*}%
Since $I^{\prime }\left( 0\right) =I^{\prime }\left( \frac{\pi }{4}\right)
=0 $, 
\begin{equation*}
\int_{\gamma }I^{\prime \prime }=0,
\end{equation*}%
so we also have 
\begin{equation*}
\int_{\gamma }e^{-2f}\mathrm{curv}\left( X,W\right) >O\left( s^{4}\right) >0.
\end{equation*}%
However, the quantity 
\begin{equation*}
s^{4}\left( D_{X}\psi \right) ^{2}+s^{4}\frac{\psi ^{2}}{\left\vert
W\right\vert ^{2}}\left( D_{X}\psi \right) ^{2}+s^{4}\frac{\psi ^{2}}{%
\left\vert W\right\vert ^{2}}D_{X}\left( \psi D_{X}\psi \right)
\end{equation*}%
can have some negative values, but by choosing $I^{\prime \prime }$ to be
sufficiently negative in the region where 
\begin{equation*}
s^{4}\left( D_{X}\psi \right) ^{2}+s^{4}\frac{\psi ^{2}}{\left\vert
W\right\vert ^{2}}\left( D_{X}\psi \right) ^{2}+s^{4}\frac{\psi ^{2}}{%
\left\vert W\right\vert ^{2}}D_{X}\left( \psi D_{X}\psi \right) <0
\end{equation*}%
we can make $e^{-2f}\mathrm{curv}\left( X,W\right) $ positive in this
region. We will have to pay for this by having $I^{\prime \prime }$ be
nonnegative on the rest of $\left[ 0,\frac{\pi }{4}\right] .$ Since 
\begin{equation*}
\int_{\gamma }s^{4}\left( D_{X}\psi \right) ^{2}+s^{4}\frac{\psi ^{2}}{%
\left\vert W\right\vert ^{2}}\left( D_{X}\psi \right) ^{2}+s^{4}\frac{\psi
^{2}}{\left\vert W\right\vert ^{2}}D_{X}\left( \psi D_{X}\psi \right) >0
\end{equation*}%
this can be achieved while keeping $e^{-2f}\mathrm{curv}\left( X,W\right) >0$
point wise.

\section{Tangential Partial Conformal Change}

There are two basic reasons why the combination of fiber scaling and a
conformal change as outlined above can not produce positive curvature on all
of the initially flat tori in the Gromoll--Meyer sphere. Before stating them
we recall that there are two families of initially flat tori in the
Gromoll--Meyer sphere, $\mathcal{F}_{\zeta }$ and $\mathcal{F}_{\xi }$ that
intersect orthogonally.

\begin{description}
\item[1] For one of the two families, $\mathcal{F}_{\zeta },$ the function 
\begin{equation*}
\frac{\left\vert W^{\mathcal{H}}\right\vert ^{2}}{\left\vert W\right\vert
^{2}}=\frac{\psi ^{2}}{\left\vert W\right\vert ^{2}}
\end{equation*}%
varies from torus to torus. The required conformal factor is $e^{2f}$ where 
\begin{equation*}
f=-\frac{s^{2}}{2\left( 1-s^{2}\right) }\frac{\psi ^{2}}{\left\vert
W\right\vert ^{2}}+s^{4}E
\end{equation*}%
and hence varies from torus to torus. A particular choice of $f$ will give
us positive curvature on \emph{some }of our tori, but for the others the
leading terms $-s^{2}\left( D_{X}\left( \psi D_{X}\psi \right) \right) ,$
and $-\left\vert W\right\vert _{g_{s}}^{2}\mathrm{Hess}f\left( X,X\right) $
will not cancel; so these tori would have curvatures of both signs. There is
no \textbf{one }conformal factor that will simultaneously make all of the
tori in $\mathcal{F}_{\zeta }$ positively curved.

\item[2] The conformal factor required to make $\mathcal{F}_{\xi }$
positively curved is different from all of the conformal factors required to
make $\mathcal{F}_{\zeta }$ positively curved.
\end{description}

Note that either of these reasons is sufficient to see that a conformal
change can not be combined with fiber scaling to put positive curvature on
all of the totally geodesic flats of the Gromoll--Meyer sphere. We have
mentioned both since both difficulties will have to be overcome.

In this section, we shall see that despite the problems mentioned above, the
results of the previous section are at least \emph{morally} correct. The
tangential partial conformal change that we describe will have the same
effect on the curvatures of the initially flat tori as an actual conformal
change--with the correct conformal factor for each torus.

Although the key function $\frac{\left\vert W^{\mathcal{H}}\right\vert ^{2}}{%
\left\vert W\right\vert ^{2}}$ varies from torus to torus on the
Gromoll--Meyer sphere, the way in which this ratio varies is rather special.
In fact, $W$ has an orthogonal decomposition 
\begin{equation*}
W=W^{\alpha }+W^{\gamma }
\end{equation*}%
where

\begin{itemize}
\item $W^{\alpha }$ is vertical for $Sp\left( 2\right) \longrightarrow
S^{4}, $

\item and $X$ and $W^{\alpha }$ span a totally geodesic flat.
\end{itemize}

Although $W^{\gamma }$ is perpendicular to $W^{\alpha }$ it is neither
vertical nor horizontal for $Sp\left( 2\right) \longrightarrow S^{4};$
however, because $W^{\alpha }$ is vertical, $W^{\mathcal{H}}=\left(
W^{\gamma }\right) ^{\mathcal{H}}.$ In particular the ratio%
\begin{equation*}
\frac{\left\vert W^{\mathcal{H}}\right\vert ^{2}}{\left\vert W^{\gamma
}\right\vert ^{2}}
\end{equation*}%
is constant on the family of tori $\mathcal{F}_{\zeta }$.

Exploiting this structure and the principle

\begin{quotation}
\emph{totally geodesic flats are preserved when the metric is changed
orthogonally to the flat,}
\end{quotation}

\noindent we will resolve the first problem by choosing the partial
conformal change to leave $g\left( W^{\alpha },\cdot \right) $ unchanged.

We show here that such a change will have the same effect on $\mathrm{curv}%
\left( X,W\right) $ as a conformal change, with $W^{\gamma }$ playing the
role of $W.$

The resolution of the second problem also exploits the principle that
totally geodesic flats are preserved when the metric is changed orthogonally
to the flat.

In the end we will make two partial conformal changes using $f_{\zeta }$ and 
$f_{\xi }.$ The $f_{\zeta }$ change will leave the metric on $\mathcal{F}%
_{\xi }$ unchanged, and the $f_{\xi }$ change will leave the metric on $%
\mathcal{F}_{\zeta }$ unchanged. Since the two families of tori intersect
orthogonally, we will be able to argue that the $f_{\zeta }$ change does not
have any effect on the curvature of $\mathcal{F}_{\xi },$and the $f_{\xi }$
change does not have any effect on the curvature of $\mathcal{F}_{\zeta }.$

The setup for our tangential partial conformal change is as follows. There
are mutually orthogonal distributions $\mathcal{X},$ $\mathcal{A},$ and $%
\mathcal{G}$ with the properties

\begin{description}
\item[1] $\mathcal{X}$ is integrable and totally geodesic.

\item[2] Any pair of vectors $Z\in \mathcal{X}$ and $U\in \mathcal{A}$ span
a totally geodesic flat torus.

\item[3] $\left[ \mathcal{X},\mathcal{A}^{\perp }\right] \subset \mathcal{A}%
^{\perp }.$

\item[4] There is a function $f$ whose gradient lies in $\mathcal{X}.$

\item[5] There is a geodesic field $X\in \mathcal{X}.$

\item[6] $\left[ X,\mathcal{G}\right] \subset \mathcal{G}.$
\end{description}

We change the metric by multiplying the lengths of all vectors in the
distribution $\mathrm{span}\left\{ X,\mathcal{G}\right\} $ by $e^{2f},$
while keeping the orthogonal complement of $\mathrm{span}\left\{ X,\mathcal{G%
}\right\} $ fixed. In particular, $g\left( \mathcal{A},\cdot \right) $ is
unchanged.

In the concrete situation we will have two functions $f_{\zeta }$ and $%
f_{\xi }.$ To accommodate this here we also assume that there is a $C^{0}$%
--small, but unspecified change to the orthogonal complement of $\mathrm{span%
}\left\{ X,\mathcal{G},\mathcal{A}\right\} .$ We call the resulting metric $%
\tilde{g}$.

In the concrete setting the splitting $W=W^{\alpha }+W^{\gamma }$ mentioned
above satisfies $W^{a}\in \mathcal{A}$ and $W^{\gamma }\in \mathcal{G}.$

We analyze here the effect of such a change on $R\left( W,X\right) X$ and $%
R\left( X,W^{\gamma }\right) W^{\gamma }.$ Since $\mathrm{span}\left\{
X,W\right\} $ is an abstraction of the zero planes in the Gromoll--Meyer
sphere, we would ideally also have formulas for $R\left( X,W\right) W;$
however, we have not succeed in making a satisfactory abstraction of this
calculation, and so have deferred it to the concrete setting.

We use the indices $z,$ $\alpha $, to denote components of the $\theta $s, $%
\omega $s and $\Omega $s corresponding to $z\in \mathcal{X}$ and $U^{\alpha
}\in \mathcal{A}$. We use $\widetilde{}$ to denote the metric quantities
with respect to $\tilde{g}$, and \textquotedblleft bar\textquotedblright\ to
denote the quantities with respect to the metric obtained from $g$ with
respect to an actual conformal change with conformal factor $e^{2f},$ e.g. $%
\bar{g}$ and $\bar{\omega}.$

\begin{proposition}
\label{z-alpha vanishing curv}For any vector $z\in \mathcal{X}$ and $%
U^{\alpha }\in \mathcal{A}$ 
\begin{eqnarray*}
\tilde{\omega}_{z}^{\alpha } &\equiv &0, \\
\tilde{\omega}_{z}^{i}\left( U^{\alpha }\right) &=&\tilde{\omega}_{\alpha
}^{i}\left( z\right) =0
\end{eqnarray*}%
for all $i.$
\end{proposition}

\begin{proof}
The last two equations are equivalent to the statement that any $z$ $\in 
\mathcal{X}$ and any $U^{\alpha }\in \mathcal{A}$ have extensions with $%
\tilde{\nabla}_{U^{\alpha }}Z=\tilde{\nabla}_{Z}U^{\alpha }=0.$ The presence
of the flat tori give us this result for $\nabla .$ For $N=Z,U^{\alpha }$ or
normal to both $Z$ and $U^{\alpha }$ we have 
\begin{eqnarray*}
0 &=&2g\left( \nabla _{U^{\alpha }}Z,N\right) \\
&=&-g\left( \left[ Z,N\right] ,U^{\alpha }\right) +g\left( \left[
N,U^{\alpha }\right] ,Z\right)
\end{eqnarray*}

Hypothesis $4$ gives us that $g\left( \left[ N,U^{\alpha }\right] ,Z\right)
=0.$ So it follows that $g\left( \left[ Z,N\right] ,U^{\alpha }\right) =0.$
This gives us 
\begin{eqnarray*}
2\tilde{g}\left( \nabla _{U^{\alpha }}Z,N\right) &=&-\tilde{g}\left( \left[
Z,N\right] ,U^{\alpha }\right) +\tilde{g}\left( \left[ N,U^{\alpha }\right]
,Z\right) \\
&=&0
\end{eqnarray*}

as desired.

The first equation is equivalent to $\left\langle \tilde{\nabla}%
_{N}Z,U^{\alpha }\right\rangle =0.$ This follows for the same reasons.
\end{proof}

\begin{proposition}
For $X$ as above and $W\in \mathrm{span}\left\{ \mathcal{A},\mathcal{G}%
\right\} $%
\begin{eqnarray*}
\tilde{\omega}_{X}^{E_{k}}\left( X\right) &=&\bar{\omega}_{X}^{E_{k}}\left(
X\right) , \\
\tilde{\omega}_{X}^{E_{k}}\left( W\right) &=&\bar{\omega}_{X}^{E_{k}}\left(
W^{\gamma }\right) , \\
\tilde{\omega}_{W}^{E_{k}}\left( X\right) &=&\bar{\omega}_{W^{\gamma
}}^{E_{k}}\left( X\right) , \\
\tilde{\omega}_{W^{\gamma }}^{E_{k}}\left( W^{\gamma }\right) &=&\bar{\omega}%
_{W^{\gamma }}^{E_{k}}\left( W^{\gamma }\right) ,
\end{eqnarray*}%
where $W^{\gamma }$ denotes the component of $W$ in $\mathcal{G}.$
\end{proposition}

\begin{proof}
By the previous proposition we have 
\begin{equation*}
\tilde{\omega}_{X}^{E_{k}}\left( W\right) =\tilde{\omega}_{X}^{E_{k}}\left(
W^{\alpha }\right) +\tilde{\omega}_{X}^{E_{k}}\left( W^{\gamma }\right) =%
\tilde{\omega}_{X}^{E_{k}}\left( W^{\gamma }\right) .
\end{equation*}

So the second equation reduces to 
\begin{equation*}
\tilde{\omega}_{X}^{E_{k}}\left( W^{\gamma }\right) =\bar{\omega}%
_{X}^{E_{k}}\left( W^{\gamma }\right) .
\end{equation*}%
Similarly, the third equation reduces to 
\begin{equation*}
\tilde{\omega}_{W}^{E_{k}}\left( X\right) =\bar{\omega}_{W^{\gamma
}}^{E_{k}}\left( X\right) .
\end{equation*}

The proofs of each of these and the first and fourth equations are
essentially the same, and boil down to the facts that 
\begin{eqnarray*}
\left[ X,X\right] &=&\left[ W^{\gamma },W^{\gamma }\right] =0,\text{ and} \\
\left[ X,W^{\gamma }\right] &\in &\mathcal{G}.
\end{eqnarray*}%
For $\tilde{\omega}_{X}^{E_{k}}\left( W^{\gamma }\right) =\bar{\omega}%
_{X}^{E_{k}}\left( W^{\gamma }\right) $ the details are 
\begin{eqnarray*}
\tilde{\omega}_{W^{\gamma }}^{E_{k}}X &=&\tilde{g}\left( \tilde{\nabla}%
_{X}W^{\gamma },E_{k}\right) \\
&=&\frac{1}{2}\left( D_{X}\tilde{g}\left( W^{\gamma },E_{k}\right)
+D_{W^{\gamma }}\tilde{g}\left( X,E_{k}\right) -D_{E_{k}}\tilde{g}\left(
X,W^{\gamma }\right) \right. \\
&&\left. +\tilde{g}\left( \left[ X,W^{\gamma }\right] ,E_{k}\right) -\tilde{g%
}\left( \left[ W^{\gamma },E_{k}\right] ,X\right) -\tilde{g}\left( \left[
X,E_{k}\right] ,W^{\gamma }\right) \right) .
\end{eqnarray*}%
For the fourth term we have $\tilde{g}\left( \left[ X,W^{\gamma }\right]
,E_{k}\right) =$ $\bar{g}\left( \left[ X,W^{\gamma }\right] ,E_{k}\right) .$
This is because of our hypothesis that $\left[ X,W^{\gamma }\right] \in 
\mathcal{G}.$ For the other terms we can also change $\tilde{g}$ to $g$ for
the same reason--that one of the vectors in the inner product is in $\mathrm{%
span}\left\{ X,\mathcal{G}\right\} .$ Thus 
\begin{eqnarray*}
\tilde{\omega}_{W^{\gamma }}^{E_{k}}X &=&\frac{1}{2}\left( D_{X}\bar{g}%
\left( W^{\gamma },E_{k}\right) +D_{W^{\gamma }}\bar{g}\left( X,E_{k}\right)
-D_{E_{k}}\bar{g}\left( X,W^{\gamma }\right) \right. \\
&&\left. +\bar{g}\left( \left[ X,W^{\gamma }\right] ,E_{k}\right) -\bar{g}%
\left( \left[ W^{\gamma },E_{k}\right] ,X\right) +\bar{g}\left( \left[
X,E_{k}\right] ,W^{\gamma }\right) \right) \\
&=&\bar{g}\left( \bar{\nabla}_{X}W^{\gamma },E_{k}\right) \\
&=&\bar{\omega}_{W^{\gamma }}^{E_{k}}X.
\end{eqnarray*}
\end{proof}

\begin{proposition}
\label{same as conformal}For any unit vector $U$%
\begin{eqnarray*}
\tilde{R}\left( W,X,X,W\right) &=&\bar{R}\left( W^{\gamma },X,X,W^{\gamma
}\right) \\
\tilde{R}\left( U,X,X,W\right) &=&\bar{R}\left( U,X,X,W^{\gamma }\right) +%
\mathrm{\ }O\left( C^{0}\right) \max \left\{ \left\vert \bar{\omega}%
_{X}^{k}\left( X\right) \right\vert ,\left\vert \bar{\omega}_{X}^{k}\left(
W^{\gamma }\right) \right\vert \right\} ,\text{ and } \\
\tilde{R}\left( U,W^{\gamma },W^{\gamma },X\right) &=&\bar{R}\left(
U,W^{\gamma },W^{\gamma },X\right) +\mathrm{\ }O\left( C^{0}\right) \max
\left\{ \left\vert \bar{\omega}_{W^{\gamma }}^{k}\left( W^{\gamma }\right)
\right\vert ,\left\vert \bar{\omega}_{X}^{k}\left( W^{\gamma }\right)
\right\vert \right\}
\end{eqnarray*}%
\newline
where $O\left( C^{0}\right) $ represents a quantity that is smaller than a
constant times the difference in the $C^{0}$ norms of $\tilde{g}$ and $\bar{g%
}.$
\end{proposition}

\begin{proof}
Using the previous two propositions we have

\begin{eqnarray*}
\tilde{R}\left( W,X,X,U\right) &=&d\tilde{\omega}_{X}^{U}\left( W,X\right)
+\dsum \tilde{\omega}_{k}^{U}\wedge \tilde{\omega}_{X}^{k}\left( W,X\right)
\\
&=&d\tilde{\omega}_{X}^{U}\left( W,X\right) +\dsum \tilde{\omega}%
_{k}^{U}\left( W\right) \bar{\omega}_{X}^{k}\left( X\right) -\tilde{\omega}%
_{k}^{U}\left( X\right) \bar{\omega}_{X}^{k}\left( W^{\gamma }\right)
\end{eqnarray*}

Since $X$ and $W^{a}$ initially span a totally geodesic flat, we can choose
our extension of $W$ so that $\left[ W^{a},X\right] =0.$ Using this, the
previous proposition and the hypothesis $\left[ W^{\gamma },X\right] \in 
\mathcal{G}$ we have

\begin{eqnarray*}
d\tilde{\omega}_{X}^{U}\left( W,X\right) &=&D_{W}\tilde{\omega}%
_{X}^{U}\left( X\right) -D_{X}\tilde{\omega}_{X}^{U}\left( W\right) -\tilde{%
\omega}_{X}^{U}\left[ W^{\gamma },X\right] \\
&=&D_{W^{\alpha }}\bar{\omega}_{X}^{U}\left( X\right) +D_{W^{\gamma }}\bar{%
\omega}_{X}^{U}\left( X\right) -D_{X}\bar{\omega}_{X}^{U}\left( W^{\gamma
}\right) -\bar{\omega}_{X}^{U}\left[ W^{\gamma },X\right] .
\end{eqnarray*}

Since $X$ is initially a geodesic field and $W^{\alpha }$ is perpendicular
to the gradient of $f$, $D_{W^{\alpha }}\bar{\omega}_{X}^{U}\left( X\right)
=0,$ as long as $U$ makes a constant angle with $\mathrm{grad}f.$ So with
such a choice of $U$ we have 
\begin{eqnarray*}
d\tilde{\omega}_{X}^{U}\left( W,X\right) &=&D_{W^{\gamma }}\bar{\omega}%
_{X}^{U}\left( X\right) -D_{X}\bar{\omega}_{X}^{U}\left( W^{\gamma }\right) -%
\bar{\omega}_{X}^{U}\left[ W^{\gamma },X\right] \\
&=&d\bar{\omega}_{X}^{U}\left( W^{\gamma },X\right) .
\end{eqnarray*}

Since $X$ is initially a geodesic field, $\bar{\nabla}_{X}X\in \mathrm{span}%
\left\{ X,\nabla f\right\} .$ Thus $\bar{\omega}_{X}^{k}\left( X\right) $ is
only nonzero for $E_{k}\in \mathcal{X}.$ By the first proposition of this
section, we have that for such $E_{k},$ $\tilde{\omega}_{k}^{U}\left(
W^{\alpha }\right) =0,$ and hence 
\begin{equation*}
\tilde{\omega}_{k}^{U}\left( W\right) \bar{\omega}_{X}^{k}\left( X\right) =%
\tilde{\omega}_{k}^{U}\left( W^{\gamma }\right) \bar{\omega}_{X}^{k}\left(
X\right) .
\end{equation*}

The Koszul formula then gives us that 
\begin{equation*}
\left\vert \tilde{\omega}_{k}^{U}\left( W^{\gamma }\right) -\bar{\omega}%
_{k}^{U}\left( W^{\gamma }\right) \right\vert \leq O\left( C^{0}\right)
\end{equation*}%
and 
\begin{equation*}
\left\vert \tilde{\omega}_{k}^{U}\left( X\right) -\bar{\omega}_{k}^{U}\left(
X\right) \right\vert \leq O\left( C^{0}\right) .
\end{equation*}

Combining these displays, give us the second equation, and a similar
argument gives us the third equation.

For $U=W$ , the first proposition of this section gives us $d\bar{\omega}%
_{X}^{W^{\alpha }}\left( W^{\gamma },X\right) =0,$ so 
\begin{equation*}
d\bar{\omega}_{X}^{W}\left( W^{\gamma },X\right) =d\bar{\omega}%
_{X}^{W^{\gamma }}\left( W^{\gamma },\zeta \right) .
\end{equation*}%
We also have to deal with 
\begin{equation*}
\dsum \tilde{\omega}_{k}^{W}\left( W\right) \bar{\omega}_{X}^{k}\left(
X\right) -\tilde{\omega}_{k}^{W}\left( X\right) \bar{\omega}_{X}^{k}\left(
W^{\gamma }\right) .
\end{equation*}%
The previous proposition gives us $\tilde{\omega}_{k}^{W}\left( X\right) =%
\bar{\omega}_{k}^{W^{\gamma }}\left( X\right) .$

We also have $\bar{\omega}_{X}^{k}\left( X\right) =\left\langle \bar{\nabla}%
_{X}X,E_{k}\right\rangle $ and $\bar{\nabla}_{X}X\in \mathrm{span}\left\{ X,%
\mathrm{grad}f\right\} .$ The previous proposition gives us $\tilde{\omega}%
_{X}^{W}\left( W\right) =\bar{\omega}_{X}^{W}\left( W^{\gamma }\right) .$
Since $\bar{\omega}_{X}^{W^{\alpha }}\left( W^{\gamma }\right) =0,$ we
conclude that $\tilde{\omega}_{X}^{W}\left( W\right) =\bar{\omega}%
_{X}^{W^{\gamma }}\left( W^{\gamma }\right) .$ The first proposition of the
section gives us 
\begin{equation*}
\tilde{\omega}_{\frac{\nabla f}{\left\vert \nabla f\right\vert }}^{W}\left(
W\right) =\tilde{\omega}_{\frac{\nabla f}{\left\vert \nabla f\right\vert }%
}^{W^{\gamma }}\left( W^{\gamma }\right)
\end{equation*}%
and the second gives us 
\begin{equation*}
\tilde{\omega}_{\frac{\nabla f}{\left\vert \nabla f\right\vert }}^{W^{\gamma
}}\left( W^{\gamma }\right) =\bar{\omega}_{\frac{\nabla f}{\left\vert \nabla
f\right\vert }}^{W^{\gamma }}\left( W^{\gamma }\right) .
\end{equation*}
So in any case we have, 
\begin{equation*}
\tilde{\omega}_{k}^{W}\left( W\right) \bar{\omega}_{X}^{k}\left( X\right) =%
\bar{\omega}_{k}^{W^{\gamma }}\left( W^{\gamma }\right) \bar{\omega}%
_{X}^{k}\left( X\right) .
\end{equation*}

Combining displays we have 
\begin{eqnarray*}
\tilde{R}\left( W,X,X,W\right) &=&d\bar{\omega}_{X}^{W^{\gamma }}\left(
W^{\gamma },X\right) +\dsum \bar{\omega}_{k}^{W^{\gamma }}\left( W^{\gamma
}\right) \bar{\omega}_{X}^{k}\left( X\right) -\omega _{k}^{W^{\gamma
}}\left( X\right) \bar{\omega}_{X}^{k}\left( W^{\gamma }\right) \\
&=&\bar{R}\left( W^{\gamma },X,X,W^{\gamma }\right) .
\end{eqnarray*}
\end{proof}

\section{Long Term Cheeger Principle}

In the presence of a group of isometries, $G,$ a method for perturbing the
metric on a manifold, $M$, of nonnegative sectional curvature is proposed in 
\cite{Cheeg}. Various special cases of this method were first studied in 
\cite{Berg3} and \cite{BourDesSent}. An exposition can be found in \cite%
{Muet}. Although this technique has been used repeatedly in the literature,
our impression is that it is not widely understood.

To understand the effect of a Cheeger deformation on the curvature of a
nonnegatively curved manifold, in our view, it is crucial to exploit the
\textquotedblleft Cheeger reparametrization\textquotedblright\ of the
Grassmannian. We will review the definition of the Cheeger reparametrization
below. For now we recall (see e.g. \cite{PetWilh1})

\begin{proposition}
\label{Cheeger nonnegative}Let $\left( M,g_{\mathrm{Cheeg}}\right) $ be a
Cheeger deformation by $G$ of the nonnegatively curved manifold $\left(
M,g\right) $. Then modulo the Cheeger reparametrization,

\begin{description}
\item[1] If a plane $P$ is positively curved with respect to $g,$ then it is
positively curved with respect to $g_{\mathrm{Cheeg}}.$

\item[2] If a plane $P$ has a nondegenerate projection onto the orbits of $G$
and \textquotedblleft corresponds\textquotedblright\ to a positively curved
plane in $G,$ then $P$ is positively curved with respect to $g_{\mathrm{Cheeg%
}}.$
\end{description}
\end{proposition}

The meaning of \textquotedblleft corresponds\textquotedblright\ will be
explained below.

In this section we will discuss a generalization of this result to manifolds
that do not necessarily have nonnegative curvature. This result is used in 
\cite{PetWilh2}.

\begin{proposition}
\label{Cheeger arbitrary}Let $\left( M,g_{\mathrm{Cheeg}}\right) $ be a
Cheeger deformation by $G$ of $\left( M,g\right) $. Then modulo the Cheeger
reparametrization,

\begin{description}
\item[1] If a plane $P$ is positively curved with respect to $g,$ then it is
positively curved with respect to $g_{\mathrm{Cheeg}}.$

\item[2] If a plane $P$ has a nondegenerate projection onto the orbits of $G$
and \textquotedblleft corresponds\textquotedblright\ to a positively curved
plane in $G,$ then $P$ is positively curved with respect to a Cheeger
deformed metric, provided the Cheeger deformation is \textquotedblleft run
for a sufficiently long time\textquotedblright .
\end{description}
\end{proposition}

The meaning of \textquotedblleft run for a sufficiently long
time\textquotedblright\ will also be explained below.

To explain these results we offer a review that is sufficient for our
purposes. None of this review is original, and in fact some of it is copied
verbatim from \cite{PetWilh1}, whose main contribution to the theory of
Cheeger deformations is expository.

If $G$ is a compact group of isometries of $M$, then we let $G$ act on $%
G\times M$ by 
\begin{equation}
g(p,m)=(pg^{-1},gm).  \label{skew action}
\end{equation}%
If we endow $G$ with a biinvariant metric and $G\times M$ with the product
metric, then the quotient of (\ref{skew action}) is a new metric on $M$. It
was observed in \cite{Cheeg}, that in a certain sense we may expect the new
metric to have more curvature and less symmetry than the original metric.
The \textquotedblleft sense\textquotedblright\ in which this is true is
modulo the Cheeger reparametrization.

The quotient map for the action (\ref{skew action}) is 
\begin{equation*}
q_{G\times M}:(g,m)\mapsto gm.
\end{equation*}%
The vertical space for $q_{G\times M}$ at $(g,m)$ is 
\begin{equation*}
\mathcal{V}_{q_{G\times M}}=\{(-k,k)\ |\ k\in \mathfrak{g}\}
\end{equation*}%
where the $-k$ in the first factor stands for the value at $g$ of the
Killing field on $G$ given by the circle action 
\begin{equation*}
(\exp (tk),g)\mapsto g\exp (-kt)
\end{equation*}%
and the $k$ in the second factor is the value of the Killing field 
\begin{equation*}
m\longmapsto \frac{d}{dt}\exp (tk)m
\end{equation*}%
on $M$ at $m$.

We recall from \cite{Cheeg}, \cite{PetWilh1} that there is a
reparametrization of the tangent space, that we will call the Cheeger
reparametrization. It is given by

\begin{equation*}
v\longmapsto Dq_{G\times M}\left( \hat{v}\right)
\end{equation*}%
where 
\begin{equation*}
\hat{v}\equiv \left( k_{v},v\right)
\end{equation*}%
is the vector tangent to $G\times M$ that is horizontal for $q_{G\times
M}:G\times M\longrightarrow M$, and projects to $v$ under $\pi _{2}:$ $%
G\times M\longrightarrow M.$

From now on we will assume that the metric on the $G$--factor in $G\times M$
is biinvariant. This means that we have only a one parameter family $\left(
M,g_{l}\right) _{l\in \mathbb{R}}$ of Cheeger deformed metrics, where $l$
denotes the scale of the biinvariant metric in $G\times M$. As $l\rightarrow
\infty ,$ $\left( M,g_{l}\right) $ converges to the metric on the $M$ factor
in $G\times M,$ so we will often call the original metric $g_{\infty }$ \cite%
{Pet}.

With an understanding of the Cheeger re-parameterization the proof of
Proposition \ref{Cheeger nonnegative} is now clear. $Dq_{G\times M}\left( 
\hat{P}\right) $ is positively curved if $\hat{P}$ is positively curved, and 
$\hat{P}$ is positively curved if its projection onto either $M$ or $G$ is
positively curved. Since the projection onto $M$ is $P,$ we get the
conclusion of Proposition \ref{Cheeger nonnegative}.

The proof of Proposition \ref{Cheeger arbitrary} is only a little harder. If 
$P$ happens to be positively curved, then so is $\hat{P}$ and hence also $%
Dq_{G\times M}\left( \hat{P}\right) .$

On the other hand, if 
\begin{eqnarray*}
\hat{P} &=&\mathrm{span}\left\{ \hat{v},\hat{w}\right\} \\
&=&\mathrm{span}\left\{ \left( k_{v},v\right) ,\left( k_{w},w\right) \right\}
\end{eqnarray*}%
when $l=1,$ then for arbitrary $l,$ 
\begin{equation*}
\hat{P}=\mathrm{span}\left\{ \left( \frac{k_{v}}{l^{2}},v\right) ,\left( 
\frac{k_{w}}{l^{2}},w\right) \right\} .
\end{equation*}%
So 
\begin{eqnarray*}
\mathrm{curv}_{\left( M,g_{l}\right) }\left( Dq_{G\times M}\left( \frac{k_{v}%
}{l^{2}},v\right) ,Dq_{G\times M}\left( \frac{k_{w}}{l^{2}},w\right) \right)
&\geq &\mathrm{curv}_{G,l}\left( \frac{k_{v}}{l^{2}},\frac{k_{w}}{l^{2}}%
\right) +\mathrm{curv}_{M}\left( v,w\right) \\
&=&\frac{1}{l^{6}}\mathrm{curv}_{G,1}\left( k_{v},k_{w}\right) +\mathrm{curv}%
_{M}\left( v,w\right)
\end{eqnarray*}%
where $\mathrm{curv}_{G,l}$ stands for the curvature with respect to the
biinvariant metric with scale $l,$ and $\mathrm{curv}_{G,1}$ stands for the
curvature with respect to the biinvariant metric with scale $1.$ Thus if $%
\mathrm{curv}_{G,1}\left( k_{v},k_{w}\right) $ happens to be positive, then
the term $\frac{1}{l^{6}}\mathrm{curv}_{G,1}\left( k_{v},k_{w}\right) $ will
dominate the term $\mathrm{curv}_{M}\left( v,w\right) ,$when $l$ is
sufficiently small, and we conclude that 
\begin{equation*}
\mathrm{curv}_{\left( M,g_{l}\right) }\left( Dq_{G\times M}\left( \frac{k_{v}%
}{l^{2}},v\right) ,Dq_{G\times M}\left( \frac{k_{w}}{l^{2}},w\right) \right)
>0.
\end{equation*}

The utility of using the Cheeger reparametrization is undeniable. As we have
seen, it provides a simple way to track changes of curvature. It also
preserves horizontal spaces of Riemannian submersions, \cite{PetWilh1}.

\begin{proposition}
\label{new horizontal space} Let $A_{H}:H\times M\longrightarrow M$ be an
action that is by isometries with respect to both $g_{\infty }$ and $g_{l}$.
Let $\mathcal{H}_{A_{H}}$ denote the distribution of vectors that are
perpendicular to the orbits of $A_{H}$.

Then $u$ is in $\mathcal{H}_{A_{H}}$ with respect to $g_{\infty }$ if and
only if $Dq_{G\times M}(\hat{u})$ is in $\mathcal{H}_{A_{H}}$ with respect
to $g_{l}$. In fact, 
\begin{equation*}
g_{\infty }\left( u,w\right) =g_{l}\left( u,Dq_{G\times M}\left( \hat{w}%
\right) \right)
\end{equation*}%
for all $u,w\in TM.$
\end{proposition}

\begin{proof}
Starting with the left and side we take the horizontal lifts to $G\times M$%
\begin{equation*}
g_{l}\left( u,Dq_{G\times M}\left( \hat{w}\right) \right) =g_{G\times
M}\left( \left( 0,u\right) -\left( 0,u\right) ^{\mathcal{V}},\hat{w}\right)
\end{equation*}%
Since $\hat{w}$ is horizontal this becomes 
\begin{eqnarray*}
g_{l}\left( u,Dq_{G\times M}\left( \hat{w}\right) \right) &=&g_{G\times
M}\left( \left( 0,u\right) ,\hat{w}\right) \\
&=&g_{\infty }\left( u,w\right) .
\end{eqnarray*}
\end{proof}

\subsection{Quadratic Nondegeneracy and Cheeger Deformations}

Cheeger deformations and the Cheeger reparametrization also play a role in
verifying the Quadratic Nondegeneracy Condition.

\begin{proposition}
Let $\left( E,g_{\infty }\right) $ be nonnegatively curved and 
\begin{equation*}
q_{H\times E}:H\times \left( E,g_{\infty }\right) \longrightarrow \left(
E,g_{l}\right)
\end{equation*}%
a Cheeger submersion. Let $M$ be as in Theorem 2.4 and be obtained as a
Riemannian submersion 
\begin{equation*}
\pi :\left( E,g_{l}\right) \longrightarrow M.
\end{equation*}

Suppose that $X$ is orthogonal to the orbits of $H$ on $\left( E,g_{\infty
}\right) ,$ $X$, $W,$ $Z$ and $V$ are $\pi $--horizontal with respect to $%
g_{\infty },$ and 
\begin{equation*}
\mathrm{span}\left\{ D\pi \circ Dq_{H\times E}\left( \hat{X}\right) ,D\pi
\circ Dq_{H\times E}\left( \hat{W}\right) \right\}
\end{equation*}%
is one of the zero curvature planes of $M.$ Then the nondegeneracy condition
holds for 
\begin{equation*}
\mathrm{span}\left\{ D\pi \circ Dq_{H\times E}\left( \hat{X}\right) +\sigma
\cdot D\pi \circ Dq_{H\times E}\left( \hat{Z}\right) ,D\pi \circ Dq_{H\times
E}\left( \hat{W}\right) +\tau \cdot D\pi \circ Dq_{H\times E}\left( \hat{V}%
\right) \right\}
\end{equation*}%
if and only if any of the following hold

\begin{itemize}
\item For the original nonnegatively curved metric $g_{\infty }$ 
\begin{equation*}
\sigma ^{2}\mathrm{curv}^{g_{\infty }}\left( Z,W\right) +2\sigma \tau \left(
R^{g_{\infty }}\left( X,W,V,Z\right) +R^{g_{\infty }}\left( X,V,W,Z\right)
\right) +\tau ^{2}\mathrm{curv}^{g_{\infty }}\left( X,V\right) >0.
\end{equation*}

\item \textrm{curv}$^{H}\left( D\pi _{H}\left( \hat{Z}\right) ,D\pi
_{H}\left( \hat{W}\right) \right) >0,$ where $\pi _{H}:H\times
G\longrightarrow H$ is projection to the $H$--factor.

\item $\left\vert \tau \cdot A_{\hat{X}}^{q_{H\times E}}\hat{V}+\sigma \cdot
A_{\hat{Z}}^{q_{H\times E}}\hat{W}\right\vert ^{2}>0.$

\item $\left\vert \tau \cdot A_{Dq_{H\times E}\left( \hat{X}\right) }^{\pi
}Dq_{H\times E}\left( \hat{V}\right) +\sigma \cdot A_{Dq_{H\times E}\left( 
\hat{Z}\right) }^{\pi }Dq_{H\times E}\left( \hat{W}\right) \right\vert
^{2}>0.$
\end{itemize}
\end{proposition}

\begin{proof}
For $g_{\infty }$ we have 
\begin{equation*}
P^{g_{\infty }}\left( \sigma ,\tau \right) =\mathrm{curv}^{g_{\infty
}}\left( X+\sigma Z,W+\tau V\right) \geq 0.
\end{equation*}%
The constant and linear terms are $0.$ So the total quadratic term is
nonnegative, otherwise $\left( E,g_{\infty }\right) $ would have a negative
curvature, near $\mathrm{span}\left\{ X,W\right\} .$

Since $\mathrm{curv}\left( D\pi \circ Dq_{H\times E}\left( \hat{X}\right)
,D\pi \circ Dq_{H\times E}\left( \hat{W}\right) \right) =0,$ it follows that 
$A_{\hat{X}}^{q_{H\times E}}\hat{W}=0.$ Therefore, writing $A$ for $%
A^{q_{H\times E}}$ and omitting the hats, 
\begin{eqnarray*}
0 &\leq &\left\langle A_{X+\sigma Z}\left( W+\tau V\right) ,A_{X+\sigma
Z}\left( W+\tau V\right) \right\rangle \\
&=&\tau ^{2}\left\langle A_{X}V,A_{X}V\right\rangle +\sigma \tau
\left\langle A_{X}V,A_{Z}W,\right\rangle +\sigma \tau ^{2}\left\langle
A_{X}V,A_{Z}V\right\rangle +\sigma \tau \left\langle
A_{Z}W,A_{X}V\right\rangle +\sigma ^{2}\tau \left\langle
A_{Z}W,A_{Z}V\right\rangle \\
&&+\sigma ^{2}\left\langle A_{Z}W,A_{Z}W\right\rangle +\sigma \tau
^{2}\left\langle A_{Z}V,A_{X}V\right\rangle +\sigma ^{2}\tau \left\langle
A_{Z}V,A_{Z}W\right\rangle +\sigma ^{2}\tau ^{2}\left\langle
A_{Z}V,A_{Z}V\right\rangle \\
&=&\tau ^{2}\left\langle A_{X}V,A_{X}V\right\rangle +2\sigma \tau
\left\langle A_{X}V,A_{Z}W\right\rangle +\sigma ^{2}\left\langle
A_{Z}W,A_{Z}W\right\rangle \\
&&+2\sigma \tau ^{2}\left\langle A_{X}V,A_{Z}V\right\rangle +2\sigma
^{2}\tau \left\langle A_{Z}W,A_{Z}V\right\rangle +\sigma ^{2}\tau
^{2}\left\langle A_{Z}V,A_{Z}V\right\rangle \\
&=&\left\vert \tau A_{X}V+\sigma A_{Z}W\right\vert ^{2}+2\sigma \tau
^{2}\left\langle A_{X}V,A_{Z}V\right\rangle +2\sigma ^{2}\tau \left\langle
A_{Z}W,A_{Z}V\right\rangle +\sigma ^{2}\tau ^{2}\left\langle
A_{Z}V,A_{Z}V\right\rangle
\end{eqnarray*}%
In particular, the effect of the Cheeger $A$--tensor on the total quadratic
term is nonnegative and given by $\left\vert \tau A_{\hat{X}}\hat{V}+\sigma
A_{\hat{Z}}\hat{W}\right\vert ^{2}.$ The same argument gives that the effect
of the $A$--tensor of $\pi $ on the total quadratic term is nonnegative and
given by $\left\vert \tau A_{Dq_{H\times E}\left( \hat{X}\right) }^{\pi
}Dq_{H\times E}\left( \hat{V}\right) +\sigma A_{Dq_{H\times E}\left( \hat{Z}%
\right) }^{\pi }Dq_{H\times E}\left( \hat{W}\right) \right\vert ^{2}.$

Because $X$ is orthogonal to the orbits of $H,$ the only quadratic term that
can be nonzero in the $H$--factor is \textrm{curv}$^{H}\left( D\pi
_{H}\left( \hat{Z}\right) ,D\pi _{H}\left( \hat{W}\right) \right) .$

Since this is also nonnegative, we have decomposed the total quadratic term
as a sum of four nonnegative quantities. If any one of these quantities is
positive, then the total quadratic term is positive. If on the other hand,
all four quantities are $0,$ then the total quadratic term is $0.$
\end{proof}

\begin{remark}
In light of the main result of [Tapp2], one might expect that the two $A$%
--tensor conditions could be omitted if $\left( E,g_{\infty }\right) $ is a
biinvariant metric on a compact Lie group. However, because the total
quadratic term is not the curvature of a plane, we are not for the moment
aware of how to prove this.
\end{remark}

\section{Curvature Compression Principle}

On the Gromoll--Meyer sphere Cheeger deformations can have a huge
quantitative impact on the formula 
\begin{equation*}
e^{-2f}\mathrm{curv}\left( X,W\right) =s^{4}\left( D_{X}\psi \right)
^{2}+s^{4}\frac{\psi ^{2}}{\left\vert W\right\vert ^{2}}\left( D_{X}\psi
\right) ^{2}+s^{4}\frac{\psi ^{2}}{\left\vert W\right\vert ^{2}}D_{X}\left(
\psi D_{X}\psi \right) -s^{4}I^{\prime \prime }\left\vert W\right\vert
^{2}+O\left( s^{6}\right)
\end{equation*}%
for $\mathrm{curv}\left( X,W\right) $ after a (Tangential Partial) conformal
change (Proposition \ref{curv(X,W)}).

In fact, by running one of our Cheeger deformations for a long time, we
shall see that the vast bulk of the first three of these terms is compressed
into a small neighborhood of the poles of $X.$

What happens to these curvatures is a lot like what happens to $\mathbb{R}%
^{2}$ under the long term Cheeger deformation by the standard $SO\left(
2\right) $--action. The metric becomes a paraboloid that is very flat except
near the fixed point, $\left( 0,0\right) ,$ where there is a lot of
curvature. In this example, the radial field, $\partial _{r},$ plays the
role of our field, $X,$ and the lengths of the circles centered about the
origin play the role of our function $\psi .$

Here we describe an abstraction of what happens on the Gromoll--Meyer
sphere, whose starting point is the fiber scaling theorem, \ref{pos int}.

Let $\pi :\left( M,g_{\infty }\right) \longrightarrow B$ be a Riemannian
submersion. Let $G=G_{1}\times G_{2}$ act isometrically on $M$ and by
symmetries of $\pi .$ Let $g_{\nu ,l}$ be the metric on $M$ obtained by
doing the Cheeger deformation with $G=G_{1}\times G_{2}$ on $M,$ where the
scale on the $G_{1}$ factor in $\left( G_{1}\times G_{2}\right) \times M$ is 
$\nu $ and the scale on the $G_{2}$ factor in $\left( G_{1}\times
G_{2}\right) \times M$ is $l.$

Because of our curvature formula, we are interested in how the length of the 
$\pi $--horizontal part, $W^{\mathcal{H}},$ of a $G_{1}$--Killing field $W$
is affected by Cheeger deforming with $\left( G_{1}\times G_{2}\right) .$ In
the Gromoll--Meyer sphere we will consider the case when $\nu $ is very
small and 
\begin{equation*}
l=O\left( \nu ^{1/3}\right) .
\end{equation*}%
So we adopt these hypotheses for our abstract framework here. We set 
\begin{eqnarray*}
\psi _{\infty } &\equiv &\left\vert W^{\mathcal{H},g_{\infty }}\right\vert
_{g_{\infty }}\text{ and } \\
\psi _{\nu ,l} &\equiv &\left\vert W^{\mathcal{H},g_{\nu ,l}}\right\vert
_{g_{\nu ,l}}.
\end{eqnarray*}%
Our goal is to obtain a formula for $\psi _{\nu ,l}$ in terms of $\psi
_{\infty }.$

\begin{lemma}
\label{psi_nu,l} Let $K_{W}^{1}$ be the Killing field on $G_{1}$ that
corresponds to $W.$ Suppose $W$ lies in the direction of the projection of $%
W^{\mathcal{H},g_{\infty }}$ onto the orbits of $G_{1}$, and that 
\begin{equation*}
\rho =\frac{1}{\left\vert K_{W}^{1}\right\vert _{\mathrm{bi}}}.
\end{equation*}%
Let $K_{W,M}^{2}$ be a vector in the direction of the projection of $W^{%
\mathcal{H},g_{\infty }}$ onto the orbit of $G_{2}$. We normalize $%
K_{W,M}^{2}$ so that $\left\vert K_{W}^{2}\right\vert _{\mathrm{bi}}^{2}=1,$
where $K_{W}^{2}$ is the corresponding Killing field on $G_{2}.$ Then%
\begin{equation*}
\psi _{\nu ,l}^{2}=\frac{\psi _{\infty }^{2}}{\rho ^{2}\frac{\psi _{\infty
}^{2}}{^{\nu ^{2}}}+\frac{\varphi _{\infty }^{4}}{l^{2}\psi _{\infty }^{2}}+1%
}
\end{equation*}%
where 
\begin{equation*}
\varphi _{\infty }^{4}\equiv \left\langle K_{W,M}^{2},W^{\mathcal{H}%
,g_{\infty }}\right\rangle ^{2}.
\end{equation*}
\end{lemma}

\begin{remark}
Note that the above definition of $K_{W,M}^{2}$ does not preclude the
possibility of $K_{W,M}^{2}$ varying from point to point.
\end{remark}

\begin{proof}
We have%
\begin{eqnarray*}
\psi _{\infty } &=&\left\vert W^{\mathcal{H},g_{\infty }}\right\vert
_{g_{\infty }} \\
&=&\left\vert \frac{\left\langle W,W^{\mathcal{H},g_{\infty }}\right\rangle 
}{\left\vert W^{\mathcal{H},g_{\infty }}\right\vert }\right\vert \\
\left\langle W,W^{\mathcal{H},g_{\infty }}\right\rangle &=&\psi _{\infty
}\left\vert W^{\mathcal{H},g_{\infty }}\right\vert =\psi _{\infty }^{2}
\end{eqnarray*}%
Because of the formula%
\begin{equation*}
g_{\infty }\left( u,w\right) =g_{\nu ,l}\left( u,Dq_{G\times M}\left( \hat{w}%
\right) \right) ,
\end{equation*}%
$Dq_{G\times M}\left( \widehat{W^{\mathcal{H},g_{\infty }}}\right) $ is
horizontal with respect to $g_{\nu ,l}.$ So setting $G_{\nu ,l}\equiv \left(
G_{1},\nu \mathrm{bi}\right) \times \left( G_{2},l\mathrm{bi}\right) $%
\begin{eqnarray*}
\psi _{\nu ,l} &\equiv &\left\vert W^{\mathcal{H},g_{\nu ,l}}\right\vert
_{g_{\nu ,l}} \\
&=&\left\vert \left\langle W,Dq_{G\times M}\left( \frac{\widehat{W^{\mathcal{%
H},g_{\infty }}}}{\left\vert \widehat{W^{\mathcal{H},g_{\infty }}}%
\right\vert _{G_{\nu ,l}\times M}}\right) \right\rangle _{g_{\nu
,l}}\right\vert \\
&=&\frac{1}{\left\vert \widehat{W^{\mathcal{H},g_{\infty }}}\right\vert
_{G_{\nu ,l}\times M}}\left\langle W,W^{\mathcal{H},g_{\infty
}}\right\rangle _{g_{\infty }} \\
&=&\frac{\psi _{\infty }^{2}}{\left\vert \widehat{W^{\mathcal{H},g_{\infty }}%
}\right\vert _{G_{\nu ,l}\times M}}
\end{eqnarray*}%
where we use the general formula 
\begin{equation*}
g_{\infty }\left( u,w\right) =g_{l}\left( u,Dq_{G\times M}\left( \hat{w}%
\right) \right)
\end{equation*}%
for the next to last equation.

We have 
\begin{eqnarray*}
\widehat{W^{\mathcal{H},g_{\infty }}} &=&\left( \frac{\left\langle W^{%
\mathcal{H},g_{\infty }},W\right\rangle }{\left\vert W\right\vert
_{M,g_{\infty }}^{2}}\frac{K_{W}^{1}}{\nu ^{2}}\frac{\left\vert W\right\vert
_{M,g_{\infty }}^{2}}{\left\vert K_{W}^{1}\right\vert _{\mathrm{bi}}^{2}},%
\hspace{0.1in}\text{ }\frac{\left\langle W^{\mathcal{H},g_{\infty
}},K_{W,M}^{2}\right\rangle }{\left\vert K_{W,M}^{2}\right\vert
_{M,g_{\infty }}^{2}}\frac{K_{W}^{2}}{l^{2}}\frac{\left\vert
K_{W,M}^{2}\right\vert _{M,g_{\infty }}^{2}}{\left\vert K_{W}^{2}\right\vert
_{\mathrm{bi}}^{2}},\hspace{0.1in}W^{\mathcal{H},g_{\infty }}\right) \\
&=&\left( \frac{K_{W}^{1}}{\nu ^{2}}\frac{\left\langle W^{\mathcal{H}%
,g_{\infty }},W\right\rangle }{\left\vert K_{W}^{1}\right\vert _{\mathrm{bi}%
}^{2}},\hspace{0.1in}\frac{K_{W}^{2}}{l^{2}}\frac{\left\langle W^{\mathcal{H}%
,g_{\infty }},K_{W,M}^{2}\right\rangle }{\left\vert K_{W}^{2}\right\vert _{%
\mathrm{bi}}^{2}},\hspace{0.1in}W^{\mathcal{H}\mathrm{,}g_{\infty }}\right)
\\
&=&\left( \rho ^{2}\left( \frac{K_{W}^{1}}{\nu ^{2}}\left\langle W^{\mathcal{%
H},g_{\infty }},W\right\rangle \right) ,\hspace{0.1in}\frac{K_{W}^{2}}{l^{2}}%
\left\langle W^{\mathcal{H},g_{\infty }},K_{W,M}^{2}\right\rangle ,\hspace{%
0.1in}W^{\mathcal{H},g_{\infty }}\right) .
\end{eqnarray*}%
This gives us 
\begin{eqnarray*}
\left\vert \widehat{W^{\mathcal{H},g_{\infty }}}\right\vert ^{2} &=&\rho ^{2}%
\frac{\left\langle W,W^{\mathcal{H},g_{\infty }}\right\rangle ^{2}}{^{\nu
^{2}}}+\frac{\left\langle K_{W,M}^{2},W^{\mathcal{H},g_{\infty
}}\right\rangle ^{2}}{^{l^{2}}}+\left\vert W^{\mathcal{H},g_{\infty
}}\right\vert ^{2} \\
&=&\rho ^{2}\frac{\psi _{\infty }^{4}}{^{\nu ^{2}}}+\frac{\varphi _{\infty
}^{4}}{^{l^{2}}}+\psi _{\infty }^{2},
\end{eqnarray*}%
and hence%
\begin{eqnarray*}
\psi _{\nu ,l}^{2} &=&\frac{\psi _{\infty }^{4}}{\rho ^{2}\frac{\psi
_{\infty }^{4}}{^{\nu ^{2}}}+\frac{\varphi _{\infty }^{4}}{^{l^{2}}}+\psi
_{\infty }^{2}} \\
&=&\frac{\psi _{\infty }^{2}}{\rho ^{2}\frac{\psi _{\infty }^{2}}{^{\nu ^{2}}%
}+\frac{\varphi _{\infty }^{4}}{l^{2}\psi _{\infty }^{2}}+1}
\end{eqnarray*}
\end{proof}

Straightforward calculation gives us formal derivatives of $\psi _{\nu ,l}$
in some unspecified direction

\begin{proposition}
\begin{equation*}
\psi _{\nu ,l}^{\prime }=\left( \psi _{\infty }^{\prime }-\frac{2\varphi
_{\infty }^{3}}{^{l^{2}}}\left( \frac{\varphi _{\infty }}{\psi _{\infty }}%
\right) ^{\prime }\right) \frac{\psi _{\nu ,l}^{3}}{\psi _{\infty }^{3}}
\end{equation*}%
\begin{eqnarray*}
\psi _{\nu ,l}^{\prime \prime } &=&\left( \psi _{\infty }^{\prime \prime }-%
\frac{6\varphi _{\infty }^{2}\varphi _{\infty }^{\prime }}{^{l^{2}}}\left( 
\frac{\varphi _{\infty }}{\psi _{\infty }}\right) ^{\prime }-\frac{2\varphi
_{\infty }^{3}}{^{l^{2}}}\left( \frac{\varphi _{\infty }}{\psi _{\infty }}%
\right) ^{\prime \prime }\right) \frac{\psi _{\nu ,l}^{3}}{\psi _{\infty
}^{3}} \\
&&-3\frac{\psi _{\infty }^{\prime }}{\psi _{\infty }}\left( \psi _{\infty
}^{\prime }-\frac{2\varphi _{\infty }^{3}}{^{l^{2}}}\left( \frac{\varphi
_{\infty }}{\psi _{\infty }}\right) ^{\prime }\right) \frac{\psi _{\nu
,l}^{3}}{\psi _{\infty }^{3}} \\
&&+3\left( \psi _{\infty }^{\prime }-\frac{2\varphi _{\infty }^{3}}{^{l^{2}}}%
\left( \frac{\varphi _{\infty }}{\psi _{\infty }}\right) ^{\prime }\right)
^{2}\frac{\psi _{\nu ,l}^{5}}{\psi _{\infty }^{6}}
\end{eqnarray*}
\end{proposition}

For the remainder of the section, we restrict our attention to a curve $%
\gamma :\left[ 0,\frac{\pi }{4}\right] \longrightarrow M$ on which 
\begin{eqnarray*}
&&\left\vert \frac{\varphi _{\infty }}{\psi _{\infty }}\right\vert \text{ is
bounded} \\
\psi _{\infty }^{\prime \prime } &\leq &0, \\
\psi _{\infty }\left( 0\right) &=&\psi _{\infty }^{\prime \prime }\left(
0\right) =0, \\
&&\left( \frac{\varphi _{\infty }}{\psi _{\infty }}\right) ^{\prime }\text{
is bounded, and} \\
\psi _{\infty }^{\prime }\left( 0\right) &=&\varphi _{\infty }^{\prime
}\left( 0\right) =1.
\end{eqnarray*}%
The next result gives a quantitative description of how $\left( \psi _{\nu
,l}^{\prime }\right) ^{2}$ is compressed as $\nu \rightarrow 0.$

\begin{proposition}
\label{Compression of psi'}For $\nu $ sufficiently small and $l=O\left( \nu
^{1/3}\right) $, we have 
\begin{equation*}
\left. \left( \psi _{\nu ,l}^{\prime }\right) ^{2}\right\vert _{\left[ 0,\nu %
\right] }\geq \frac{97}{100}\frac{1}{\rho ^{2}+1},
\end{equation*}%
\begin{equation*}
\left. \left( \psi _{\nu ,l}^{\prime }\right) ^{2}\right\vert _{\left[ \nu
^{\beta },\frac{\pi }{4}\right] }\leq O\left( \nu ^{\frac{14}{3}-6\beta
}\right) ,
\end{equation*}%
for any fixed $\beta <\frac{7}{9}.$
\end{proposition}

\begin{remark}
Keeping in mind that $\rho $ is a fixed \textquotedblleft background
constant\textquotedblright , $\rho =\frac{1}{\left\vert K_{W}^{1}\right\vert
_{\mathrm{bi}}},$ that is independent of $\nu ,$ these formulas tell us that 
$\left( \psi _{\nu ,l}^{\prime }\right) ^{2}$ is large on the small interval 
$\left[ 0,\nu \right] ,$ and then rapidly becomes very small, its generic
order being $\nu ^{\frac{14}{3}}.$
\end{remark}

\begin{proof}
Setting $D^{2}=\rho ^{2}\frac{\psi _{\infty }^{2}}{\nu ^{2}}+\frac{\varphi
_{\infty }^{4}}{l^{2}\psi _{\infty }^{2}}+1$ we have 
\begin{eqnarray*}
\frac{\psi _{\nu ,l}^{2}}{\psi _{\infty }^{2}} &=&\frac{1}{D^{2}} \\
&=&\frac{\nu ^{2}}{\rho ^{2}\psi _{\infty }^{2}+\frac{\nu ^{2}\varphi
_{\infty }^{4}}{l^{2}\psi _{\infty }^{2}}+\nu ^{2}}
\end{eqnarray*}%
So%
\begin{equation*}
\left( \psi _{\nu ,l}^{\prime }\right) ^{2}=\left( \psi _{\infty }^{\prime }-%
\frac{2\varphi _{\infty }^{3}}{^{l^{2}}}\left( \frac{\varphi _{\infty }}{%
\psi _{\infty }}\right) ^{\prime }\right) ^{2}\frac{1}{D^{6}}
\end{equation*}%
On $\left[ 0,\nu \right] $ our hypotheses imply that $\frac{2\varphi
_{\infty }^{3}}{^{l^{2}}}\left( \frac{\varphi _{\infty }}{\psi _{\infty }}%
\right) ^{\prime }$ is much smaller than $\psi _{\infty }^{\prime };$ so on $%
\left[ 0,\nu \right] $ we have 
\begin{eqnarray*}
\left( \psi _{\nu ,l}^{\prime }\right) ^{2} &\geq &\frac{99}{100}\left( \psi
_{\infty }^{\prime }\right) ^{2}\frac{1}{D^{6}} \\
&\geq &\frac{98}{100}\frac{1}{\rho ^{2}\frac{\psi _{\infty }^{2}}{^{\nu ^{2}}%
}+1} \\
&\geq &\frac{97}{100}\frac{1}{\rho ^{2}\frac{t^{2}}{^{\nu ^{2}}}+1} \\
&\geq &\frac{97}{100}\frac{\nu ^{2}}{\rho ^{2}t^{2}+\nu ^{2}} \\
&\geq &\frac{97}{100}\frac{1}{\rho ^{2}+1}
\end{eqnarray*}%
We get the upper estimate on $\left[ \nu ^{\beta },\frac{\pi }{4}\right] $
by using 
\begin{eqnarray*}
\psi _{\infty }^{\prime } &\leq &1, \\
\frac{2\varphi _{\infty }^{3}}{^{l^{2}}}\left( \frac{\varphi _{\infty }}{%
\psi _{\infty }}\right) ^{\prime } &=&O\left( \frac{\psi _{\infty }^{3}}{%
l^{2}}\right)
\end{eqnarray*}%
and 
\begin{eqnarray*}
\frac{1}{D^{6}} &=&\left( \frac{\nu ^{2}}{\rho ^{2}\psi _{\infty }^{2}+\frac{%
\varphi _{\infty }^{4}}{\psi _{\infty }^{2}}\frac{\nu ^{2}}{l^{2}}+\nu ^{2}}%
\right) ^{3} \\
&\leq &O\left( \frac{\nu ^{2}}{\rho ^{2}\psi _{\infty }^{2}+\psi _{\infty
}^{2}\frac{\nu ^{2}}{l^{2}}+\nu ^{2}}\right) ^{3} \\
&\leq &O\left( \frac{\nu ^{2}}{\rho ^{2}\nu ^{2\beta }+\nu ^{2}}\right) ^{3}
\\
&\leq &\left( \frac{\nu ^{2}}{\rho ^{2}\nu ^{2\beta }}\right) ^{3} \\
&\leq &\frac{1}{\rho ^{2}}\nu ^{6\left( 1-\beta \right) }
\end{eqnarray*}%
So%
\begin{eqnarray*}
\left. \left( \psi _{\nu ,l}^{\prime }\right) ^{2}\right\vert _{\left[ \nu
^{\beta },\frac{\pi }{4}\right] } &\leq &O\left( \frac{1}{\rho ^{2}}\frac{%
\nu ^{6\left( 1-\beta \right) }}{l^{4}}\right) \\
&=&O\left( \frac{1}{\rho ^{2}}\nu ^{6\left( 1-\beta \right) -\frac{4}{3}%
}\right) \\
&=&O\left( \nu ^{\frac{14}{3}-6\beta }\right)
\end{eqnarray*}
\end{proof}

The results in the remainder of this section will be used in \cite{PetWilh2}%
, but not in this paper.

\begin{lemma}
In addition to the assumptions above suppose that 
\begin{equation*}
\mathrm{curv}^{M}\left( X,W^{\mathcal{H},g_{\infty }}\right) \geq C_{1}\psi
_{\infty }^{2}
\end{equation*}%
for some positive constant $C_{1}.$

Let $\gamma :\left[ 0,\pi /4\right] \longrightarrow M$ be as above, then for
any $\beta >0$ 
\begin{equation*}
-\left( D_{X}\left( \psi _{\nu ,l}D_{X}\psi _{\nu ,l}\right) \right) >0,
\end{equation*}%
provided $t\geq \frac{\nu }{\sqrt{3}\rho }+\beta \nu ,$ and $\nu $ is
sufficiently small.
\end{lemma}

\begin{proof}
Because the second derivative of $\psi _{\nu ,l}$ is so complicated, we
divide the proof of the first inequality into the case where $t\geq O\left(
\nu ^{1/2}\right) $ and the case where $t\leq O\left( \nu ^{1/2}\right) .$
Since $\psi _{\infty }\equiv \left\vert W^{\mathcal{H},g_{\infty
}}\right\vert _{g_{\infty }}$ and $\psi _{\nu ,l}\equiv \left\vert W^{%
\mathcal{H},g_{\nu ,l}}\right\vert _{g_{\nu ,l}}$ we have using Proposition %
\ref{R^B ( H, X) X} 
\begin{eqnarray*}
-\psi _{\nu ,l}\psi _{\nu ,l}^{\prime \prime } &=&\mathrm{curv}^{B}\left(
X,W^{\mathcal{H},g_{\nu ,l}}\right) \\
&\geq &\mathrm{curv}^{M}\left( X,W^{\mathcal{H},g_{\nu ,l}}\right) \\
&=&\psi _{\nu ,l}^{2}\mathrm{curv}^{M}\left( X,\frac{W^{\mathcal{H},g_{\nu
,l}}}{\left\vert W^{\mathcal{H},g_{\nu ,l}}\right\vert ^{2}}\right) \\
&=&\psi _{\nu ,l}^{2}\mathrm{curv}^{M}\left( X,Dq_{G\times M}\left( \frac{%
\widehat{W^{\mathcal{H}\mathrm{,}g_{\infty }}}}{\left\vert \widehat{W^{%
\mathcal{H}\mathrm{,}g_{\infty }}}\right\vert ^{2}}\right) \right) \\
&=&\frac{\psi _{\nu ,l}^{2}}{\left\vert \widehat{W^{\mathcal{H},g_{\infty }}}%
\right\vert ^{2}}\mathrm{curv}^{M}\left( X,Dq_{G\times M}\left( \widehat{W^{%
\mathcal{H}\mathrm{,}g_{\infty }}}\right) \right) \\
&\geq &\frac{\psi _{\nu ,l}^{2}}{\left\vert \widehat{W^{\mathcal{H}%
,g_{\infty }}}\right\vert _{G_{\nu ,l}\times M}^{2}}C_{1}\psi _{\infty }^{2}
\end{eqnarray*}

Using $\left\vert \widehat{W^{\mathcal{H},g_{\infty }}}\right\vert ^{2}=\psi
_{\infty }^{2}D^{2}$ and $\psi _{\nu ,l}^{2}=\frac{\psi _{\infty }^{2}}{D^{2}%
}$ we have%
\begin{eqnarray*}
-\psi _{\nu ,l}\psi _{\nu ,l}^{\prime \prime } &\geq &\frac{\psi _{\infty
}^{2}}{D^{2}}\frac{1}{\psi _{\infty }^{2}D^{2}}C_{1}\psi _{\infty }^{2} \\
&=&C_{1}\frac{\psi _{\infty }^{2}}{D^{4}}
\end{eqnarray*}%
and 
\begin{equation*}
\left( \psi _{\nu ,l}^{\prime }\right) ^{2}=\frac{\left( \psi _{\infty
}^{\prime }-\frac{2\varphi _{\infty }^{3}}{^{l^{2}}}\left( \frac{\varphi
_{\infty }}{\psi _{\infty }}\right) ^{\prime }\right) ^{2}}{D^{6}}
\end{equation*}%
So it would be enough to prove%
\begin{equation*}
\left( \psi _{\infty }^{\prime }-\frac{2\varphi _{\infty }^{3}}{^{l^{2}}}%
\left( \frac{\varphi _{\infty }}{\psi _{\infty }}\right) ^{\prime }\right)
^{2}\leq C_{1}\psi _{\infty }^{2}D^{2}
\end{equation*}

For $t\geq O\left( \nu ^{1/2}\right) $ 
\begin{equation*}
\left( \psi _{\infty }^{\prime }-\frac{2\varphi _{\infty }^{3}}{^{l^{2}}}%
\left( \frac{\varphi _{\infty }}{\psi _{\infty }}\right) ^{\prime }\right)
^{2}\leq \left( \psi _{\infty }^{\prime }\right) ^{2}+O\left( \frac{t^{3}}{%
l^{2}}\right) +O\left( \frac{t^{6}}{l^{4}}\right)
\end{equation*}%
and since $D^{2}=\rho ^{2}\frac{\psi _{\infty }^{2}}{^{\nu ^{2}}}+\frac{%
\varphi _{\infty }^{4}}{l^{2}\psi _{\infty }^{2}}+1,$ 
\begin{equation*}
C_{2}t^{2}\left( 1+\rho ^{2}\frac{t^{2}}{\nu ^{2}}\right) \leq C_{1}\psi
_{\infty }^{2}D^{2},
\end{equation*}%
for another positive constant $C_{2}.$ So the desired inequality would
follow from 
\begin{equation*}
\left( \psi _{\infty }^{\prime }\right) ^{2}\leq C_{2}t^{2}\left( 1+\rho ^{2}%
\frac{t^{2}}{\nu ^{2}}\right)
\end{equation*}%
or 
\begin{equation*}
1\leq O\left( \frac{t^{4}}{\nu ^{2}}\right)
\end{equation*}%
or%
\begin{equation*}
t\geq O\left( \nu ^{1/2}\right) .
\end{equation*}

For $t\leq O\left( \nu ^{1/2}\right) ,$

\begin{eqnarray*}
\left( \psi _{\nu ,l}^{\prime }\right) ^{2} &=&\left( \psi _{\infty
}^{\prime }-\frac{2\varphi _{\infty }^{3}}{^{l^{2}}}\left( \frac{\varphi
_{\infty }}{\psi _{\infty }}\right) ^{\prime }\right) ^{2}\frac{\psi _{\nu
,l}^{6}}{\psi _{\infty }^{6}} \\
&\leq &\left( \psi _{\infty }^{\prime }-O\left( \frac{\nu ^{3/2}}{^{l^{2}}}%
\right) \right) ^{2}\frac{\psi _{\nu ,l}^{6}}{\psi _{\infty }^{6}},
\end{eqnarray*}%
\begin{eqnarray*}
\left\vert \psi _{\nu ,l}\psi _{\nu ,l}^{\prime \prime }\right\vert &\geq
&\left\vert \left( \psi _{\infty }^{\prime \prime }-\frac{6\varphi _{\infty
}^{2}\varphi _{\infty }^{\prime }}{^{l^{2}}}\left( \frac{\varphi _{\infty }}{%
\psi _{\infty }}\right) ^{\prime }-\frac{2\varphi _{\infty }^{3}}{^{l^{2}}}%
\left( \frac{\varphi _{\infty }}{\psi _{\infty }}\right) ^{\prime \prime
}\right) \frac{\psi _{\nu ,l}^{4}}{\psi _{\infty }^{3}}\right. \\
&&-3\frac{\psi _{\infty }^{\prime }}{\psi _{\infty }}\left( \psi _{\infty
}^{\prime }-\frac{2\varphi _{\infty }^{3}}{^{l^{2}}}\left( \frac{\varphi
_{\infty }}{\psi _{\infty }}\right) ^{\prime }\right) \frac{\psi _{\nu
,l}^{4}}{\psi _{\infty }^{3}} \\
&&\left. +3\left( \psi _{\infty }^{\prime }-\frac{2\varphi _{\infty }^{3}}{%
^{l^{2}}}\left( \frac{\varphi _{\infty }}{\psi _{\infty }}\right) ^{\prime
}\right) ^{2}\frac{\psi _{\nu ,l}^{6}}{\psi _{\infty }^{6}}\right\vert \\
&\geq &\left\vert \left( \psi _{\infty }^{\prime \prime }-O\left( \frac{t^{2}%
}{^{l^{2}}}\right) -O\left( \frac{t^{3}}{^{l^{2}}}\right) \right) \frac{\psi
_{\nu ,l}^{4}}{\psi _{\infty }^{3}}\right. \\
&&-3\frac{\psi _{\infty }^{\prime }}{\psi _{\infty }}\left( \psi _{\infty
}^{\prime }-O\left( \frac{t^{3}}{^{l^{2}}}\right) \right) \frac{\psi _{\nu
,l}^{4}}{\psi _{\infty }^{3}} \\
&&\left. +3\left( \psi _{\infty }^{\prime }-O\left( \frac{t^{3}}{^{l^{2}}}%
\right) \right) ^{2}\frac{\psi _{\nu ,l}^{6}}{\psi _{\infty }^{6}}\right\vert
\end{eqnarray*}%
Since $t\leq O\left( \nu ^{1/2}\right) $ we conclude that 
\begin{equation*}
\left\vert \psi _{\nu ,l}\psi _{\nu ,l}^{\prime \prime }\right\vert \geq
\left\vert \left( \psi _{\infty }^{\prime \prime }\right) \frac{\psi _{\nu
,l}^{4}}{\psi _{\infty }^{3}}-3\left( \psi _{\infty }^{\prime }\right)
^{2}\left( \frac{\psi _{\nu ,l}^{4}}{\psi _{\infty }^{4}}-\frac{\psi _{\nu
,l}^{6}}{\psi _{\infty }^{6}}\right) \right\vert +O,
\end{equation*}%
where $``O\textquotedblright $\ stands for a quantity that is too small to
matter.

Recalling that $\frac{\psi _{\infty }^{2}}{\psi _{\nu ,l}^{2}}=D^{2}=\rho
^{2}\frac{\psi _{\infty }^{2}}{^{\nu ^{2}}}+\frac{\varphi _{\infty }^{4}}{%
l^{2}\psi _{\infty }^{2}}+1$ we have%
\begin{eqnarray*}
\left( \frac{\psi _{\nu ,l}^{4}}{\psi _{\infty }^{4}}-\frac{\psi _{\nu
,l}^{6}}{\psi _{\infty }^{6}}\right) &=&\frac{1}{D^{4}}-\frac{1}{D^{6}} \\
&=&\frac{D^{2}-1}{D^{6}} \\
&=&\frac{\rho ^{2}\frac{\psi _{\infty }^{2}}{^{\nu ^{2}}}}{D^{6}}+O \\
&=&\frac{\rho ^{2}\psi _{\infty }^{2}}{\nu ^{2}D^{6}}+O
\end{eqnarray*}%
Since $\psi _{\infty }^{\prime \prime }\left( 0\right) =0,$it follows that
for $t\leq O\left( \nu ^{1/2}\right) ,$%
\begin{equation*}
\left\vert \psi _{\nu ,l}\psi _{\nu ,l}^{\prime \prime }\right\vert \geq
3\left( \psi _{\infty }^{\prime }\right) ^{2}\frac{\rho ^{2}\psi _{\infty
}^{2}}{\nu ^{2}D^{6}}+O.
\end{equation*}

Thus our total derivative is positive when 
\begin{eqnarray*}
\left( \psi _{\infty }^{\prime }\right) ^{2}\left( \frac{\psi _{\nu ,l}^{6}}{%
\psi _{\infty }^{6}}\right) &\leq &3\left( \psi _{\infty }^{\prime }\right)
^{2}\frac{\rho ^{2}\psi _{\infty }^{2}}{\nu ^{2}D^{6}}+O,\text{ or} \\
\left( \frac{1}{D^{6}}\right) &\leq &3\rho ^{2}\frac{\psi _{\infty }^{2}}{%
\nu ^{2}D^{6}}+O,\text{ or} \\
\nu &\leq &\sqrt{3}\rho \psi _{\infty }+O,\text{ or }
\end{eqnarray*}%
Since $\psi _{\infty }^{\prime }\left( 0\right) =1,$ $\psi _{\infty
}^{\prime \prime }\left( 0\right) =0$, this is equivalent to%
\begin{equation*}
t\geq \frac{\nu }{\sqrt{3}\rho }+O
\end{equation*}
\end{proof}

\begin{lemma}
\begin{equation}
\left\vert \frac{\psi _{\nu ,l}}{\psi _{\nu ,l}^{\prime \prime }}\left[ \psi
_{\nu ,l}\psi _{\nu ,l}^{\prime }\right] ^{\prime }\right\vert \leq \max
\left\{ \psi _{\nu ,l}^{2},\frac{\nu ^{2}}{3\rho ^{2}}\right\} .
\label{Peter's est}
\end{equation}
\end{lemma}

\begin{proof}
From the previous result{\Huge \ }we have that for $t\geq \frac{\nu }{\sqrt{3%
}\rho }+O$, $\left\vert \left[ \psi _{\nu ,l}\psi _{\nu ,l}^{\prime }\right]
^{\prime }\right\vert \leq \left\vert \psi _{\nu ,l}\psi _{\nu ,l}^{\prime
\prime }\right\vert ,$ so

\begin{equation*}
\left\vert \frac{\psi _{\nu ,l}}{\psi _{\nu ,l}^{\prime \prime }}\left[ \psi
_{\nu ,l}\psi _{\nu ,l}^{\prime }\right] ^{\prime }\right\vert \leq \psi
_{\nu ,l}^{2}
\end{equation*}

For some $t\leq \frac{\nu }{\sqrt{3}\rho }+O$ the above inequality fails,
but then we have 
\begin{equation*}
\left\vert \frac{\psi _{\nu ,l}}{\psi _{\nu ,l}^{\prime \prime }}\left[ \psi
_{\nu ,l}\psi _{\nu ,l}^{\prime }\right] ^{\prime }\right\vert \leq
\left\vert \frac{\psi _{\nu ,l}}{\psi _{\nu ,l}^{\prime \prime }}\left[ \psi
_{\nu ,l}^{\prime }\right] ^{2}\right\vert
\end{equation*}%
Estimating as in Proposition \ref{Compression of psi'} we have that for $%
t\leq \frac{\nu }{\sqrt{3}\rho }+O$ 
\begin{equation*}
\left[ \psi _{\nu ,l}^{\prime }\right] ^{2}\leq \left[ \left( \psi _{\infty
}^{\prime }\right) \frac{\psi _{\nu ,l}^{3}}{\psi _{\infty }^{3}}\right]
^{2}+O
\end{equation*}%
and%
\begin{eqnarray*}
\psi _{\nu ,l}^{\prime \prime } &=&\psi _{\infty }^{\prime \prime }\frac{%
\psi _{\nu ,l}^{3}}{\psi _{\infty }^{3}}-3\left( \psi _{\infty }^{\prime
}\right) ^{2}\left( \frac{\psi _{\nu ,l}^{3}}{\psi _{\infty }^{4}}-\frac{%
\psi _{\nu ,l}^{5}}{\psi _{\infty }^{6}}\right) +O \\
&=&-3\left( \psi _{\infty }^{\prime }\right) ^{2}\left( \frac{\psi _{\nu
,l}^{3}}{\psi _{\infty }^{4}}-\frac{\psi _{\nu ,l}^{5}}{\psi _{\infty }^{6}}%
\right) +O
\end{eqnarray*}%
Using $D^{2}=\rho ^{2}\frac{\psi _{\infty }^{2}}{^{\nu ^{2}}}+\frac{\varphi
_{\infty }^{4}}{l^{2}\psi _{\infty }^{2}}+1$ we have%
\begin{eqnarray*}
\left( \frac{\psi _{\nu ,l}^{3}}{\psi _{\infty }^{4}}-\frac{\psi _{\nu
,l}^{5}}{\psi _{\infty }^{6}}\right) &=&\frac{1}{\psi _{\infty }}\left( 
\frac{1}{D^{3}}-\frac{1}{D^{5}}\right) \\
&=&\frac{1}{\psi _{\infty }}\left( \frac{D^{2}-1}{D^{5}}\right) \\
&=&\frac{\rho ^{2}}{\nu ^{2}}\left( \frac{\psi _{\infty }}{D^{5}}\right) +O
\end{eqnarray*}

So 
\begin{eqnarray*}
\left\vert \frac{\psi _{\nu ,l}}{\psi _{\nu ,l}^{\prime \prime }}\left[ \psi
_{\nu ,l}\psi _{\nu ,l}^{\prime }\right] ^{\prime }\right\vert &\leq
&\left\vert \frac{\psi _{\nu ,l}}{\psi _{\nu ,l}^{\prime \prime }}\left[
\psi _{\nu ,l}^{\prime }\right] ^{2}\right\vert \\
&\leq &\frac{1}{3\left( \psi _{\infty }^{\prime }\right) ^{2}\frac{\rho ^{2}%
}{\nu ^{2}}\left( \frac{\psi _{\infty }}{D^{5}}\right) }\left[ \left( \psi
_{\infty }^{\prime }\right) ^{2}\frac{\psi _{\nu ,l}^{7}}{\psi _{\infty }^{6}%
}\right] \\
&=&\frac{\nu ^{2}D^{5}}{3\rho ^{2}\left( \psi _{\infty }\right) }\left[ 
\frac{1}{D^{6}}\right] \psi _{\nu ,l} \\
&=&\frac{\nu ^{2}D^{5}}{3\rho ^{2}\left( \psi _{\infty }\right) }\left[ 
\frac{1}{D^{6}}\right] \frac{\psi _{\infty }}{D} \\
&=&\frac{\nu ^{2}}{3\rho ^{2}}\left[ \frac{1}{D^{2}}\right]
\end{eqnarray*}%
Since $D^{2}\geq 1,$ we get the desired inequality.
\end{proof}

\section{Synergy}

In Section 4, we detailed an abstract setting for which fiber scaling
produces integrally positive curvature on initially flat totally geodesic
tori. We also explained how, with a few extra hypotheses, this deformation
can be combined with a conformal change to produce positive curvature on a
single initially flat torus. In Section 5, we described an abstract
framework that will allow to use a tangential partial conformal change to
put positive curvature on all the initially flat tori in the Gromoll-Meyer
sphere, simultaneously. However, we are not aware of any way to combine
Cheeger deformations, fiber scaling and tangential partial conformal changes
to put positive curvature on the Gromoll--Meyer sphere. The problem is that
these deformations only produce positive curvature to higher order on the
initially flat tori. In principle, such a deformation could produce positive
curvature, but much more needs to be verified. As far as we can tell this
verification must fail for the Gromoll--Meyer sphere.

We described in Sections 2 and 3 a method, called orthogonal partial
conformal change, that will allow us to change the metric on the
Gromoll--Meyer sphere to one that

\begin{itemize}
\item has nonnegative curvature,

\item the same zero curvatures,
\end{itemize}

\noindent and to which we will be able to apply a combination of Cheeger
deformations, fiber scaling, and partial conformal changes and get positive
curvature.

In this section, we discuss in an abstract setting, how the orthogonal
partial conformal change of Sections 2 and 3 will play a role in making our
problem more solvable. This will involve a synergy between the curvature
compression principle, fiber scaling, and the orthogonal partial conformal
change.

To allow for a slightly less intertwined exposition we will explain this
synergy as it applies to a single torus. This will allow us to use a
conformal change in place of the tangential conformal change.

The addition of the orthogonal partial conformal change will aid us in
verifying the positivity of the curvatures of planes of the form%
\begin{equation*}
\mathrm{span}\left\{ X,W+\tau V\right\} ,
\end{equation*}%
where $V$ is perpendicular to $X$, $W$, and $W^{\mathcal{H}},$ and $\tau \in 
\mathbb{R}.$ It is necessary that such planes have positive curvature, but
of course it is not sufficient.

The curvature of $\mathrm{span}\left\{ X,W+\tau V\right\} $ is a quadratic
polynomial in $\tau $%
\begin{equation*}
Q\left( \tau \right) =\mathrm{curv}\left( X,W\right) +2\tau R\left(
W,X,X,V\right) +\tau ^{2}\mathrm{curv}\left( X,V\right)
\end{equation*}%
whose minimum value is 
\begin{equation*}
\mathrm{curv}\left( X,W\right) -\frac{R\left( W,X,X,V\right) ^{2}}{\mathrm{%
curv}\left( X,V\right) }.
\end{equation*}

\begin{proposition}
Let $M$ be nonnegatively curved and let $X,W$ satisfy the hypotheses of
Section $4$. After scaling the fibers of the Riemannian submersion $\pi $
and performing the conformal change described in subsection 4.1,%
\begin{eqnarray*}
\mathrm{curv}\left( X,W\right) -\frac{R\left( W,X,X,V\right) ^{2}}{\mathrm{%
curv}\left( X,V\right) } &=&s^{4}\left( D_{X}\psi \right) ^{2}+s^{4}\frac{%
\psi ^{2}}{\left\vert W\right\vert ^{2}}\left( D_{X}\psi \right) ^{2}+s^{4}%
\frac{\psi ^{2}}{\left\vert W\right\vert ^{2}}D_{X}\left( \psi D_{X}\psi
\right) -s^{4}I^{\prime \prime }\left\vert W\right\vert ^{2} \\
&&-s^{4}\left( D_{X}\psi \right) ^{2}\frac{\left\langle \frac{W^{\mathcal{H}}%
}{\left\vert W^{\mathcal{H}}\right\vert },A_{X}V\right\rangle ^{2}}{\mathrm{%
curv}\left( X,V\right) }+O\left( s^{6}\right) ,
\end{eqnarray*}%
provided $V$ is perpendicular to $X$, $W$, and $W^{\mathcal{H}},$ and $%
\mathrm{Hess}^{g_{s}}\left( f\right) \left( W,V\right) =0$ where $f$ is as
in Section $4.$
\end{proposition}

\begin{remark}
In our application $\mathrm{Hess}^{g_{s}}\left( f\right) \left( W,V\right) $
will not be $0,$ but it will be too small to matter.
\end{remark}

\begin{proof}
According to Proposition \ref{curv(X,W)}, the first four terms are just $%
\mathrm{curv}\left( X,W\right) .$

Because $X$ and $W$ are initially tangent to a totally geodesic flat in a
nonnegatively curved manifold our initial curvature, $R^{\mathrm{old}},$
satisfies%
\begin{equation*}
R^{\mathrm{old}}\left( W,X\right) X=0.
\end{equation*}%
In particular, 
\begin{equation*}
R^{\mathrm{old}}\left( W,X,X,V\right) =0
\end{equation*}

Our hypotheses on $V$ combined with Lemma \ref{Abstract (1,3) tensors} give
us that after fiber scaling 
\begin{equation*}
R^{g_{s}}\left( W,X,X,V\right) ^{2}=s^{4}\left( D_{X}\psi \right)
^{2}\left\langle \frac{W^{\mathcal{H}}}{\left\vert W^{\mathcal{H}%
}\right\vert },A_{X}V\right\rangle ^{2}.
\end{equation*}%
It remains to verify that this formula continues to hold after our conformal
change. After the conformal change we have%
\begin{eqnarray*}
e^{-2f}\left\langle R^{\mathrm{new}}\left( W,X\right) X,V\right\rangle
&=&\left\langle R^{g_{s}}\left( W,X\right) X,V\right\rangle \\
&&-g_{s}\left( W,V\right) \mathrm{Hess}^{g_{s}}\left( f\right) \left(
X,X\right) -g_{s}\left( X,X\right) \mathrm{Hess}^{g_{s}}\left( f\right)
\left( W,V\right) \\
&&+g_{s}\left( X,V\right) \mathrm{Hess}^{g_{s}}\left( f\right) \left(
W,X\right) \\
&&+g_{s}\left( W,V\right) D_{X}fD_{X}f-g_{s}\left( X,X\right) g_{s}\left(
W,V\right) \left\vert \mathrm{grad}f\right\vert ^{2}
\end{eqnarray*}

Our hypotheses about $V$ immediately simplifies this to 
\begin{equation*}
e^{-2f}\left\langle R^{\mathrm{new}}\left( W,X\right) X,V\right\rangle
=\left\langle R^{g_{s}}\left( W,X\right) X,V\right\rangle
\end{equation*}
\end{proof}

In addition to the hypotheses of the previous Proposition we assume the
following.

\begin{itemize}
\item We have the set up for the orthogonal partial conformal change of
Section 3, with $X$ and $W$ tangent to one of the flats $S$ and $V$ tangent
to the distribution $\mathcal{O}$ of Section 3.

\item There is a $G_{1}\times G_{2}$ action on $M$ as in Lemma \ref{psi_nu,l}%
, and the action of $G_{1}$ coincides with that of $G$ from the fiber
scaling of section 4.
\end{itemize}

We now carry out metric deformations in the following order.

\begin{itemize}
\item Cheeger deform with $G_{1}=G$ and with the Cheeger parameter $\nu $
being small.

\item Perform the \emph{orthogonal partial conformal change }with $\varphi $
as in Section 3.

\item Scale the fibers of the Riemannian submersion $\pi :\left(
M,g_{0}\right) \longrightarrow B,$ as in Section 4, and

\item perform a conformal change with conformal factor $f$ as in Subsection
4.1.
\end{itemize}

As usual we call the initial metric $g$ and the final metric $\tilde{g}.$
The metric obtained by omitting the orthogonal partial conformal change will
be called $\bar{g}.$

\begin{remark}
We have chosen to explain the synergy only for a single torus. Because of
this our final deformation can be an actual conformal change rather than the
tangential conformal change described in Section 5. The abstract framework
of Section 5 will allow us to achieve the same results on all of the
initially flat tori of the Gromoll--Meyer sphere using a tangential
conformal change.
\end{remark}

\begin{lemma}
In addition to the hypotheses above assume that the ratio 
\begin{equation*}
\frac{g\left( \frac{W^{\mathcal{H}}}{\left\vert W^{\mathcal{H}}\right\vert }%
,A_{X}V\right) ^{2}}{\mathrm{curv}^{g}\left( X,V\right) }\leq C
\end{equation*}%
for all $\nu .$

There is a function $\varkappa :\left( 0,1\right) \longrightarrow \mathbb{R}%
_{+}$ with $\lim_{t\rightarrow 0}\varkappa \left( t\right) =0$ so that for
all $\tau \in \mathbb{R}$ 
\begin{equation}
Q\left( \tau \right) \equiv \mathrm{curv}^{\tilde{g}}\left( X,W+\tau
V\right) \geq \left( 1-\varkappa \left( \nu \right) \right) \mathrm{curv}^{%
\tilde{g}}\left( X,W\right) >0,  \label{kappa of nu}
\end{equation}%
provided that the $\varphi $ used in the orthogonal partial conformal change
is chosen appropriately.
\end{lemma}

\begin{remark}
The reader should note that this lemma not only shows that these curvatures
are positive, but also shows that we can make them as close as we like to $%
\mathrm{curv}^{\tilde{g}}\left( X,W\right) .$ This will be important for our
computations on the Gromoll--Meyer sphere.

Without the orthogonal partial conformal change we still get an estimate
that roughly looks like%
\begin{equation*}
Q\left( \tau \right) \equiv \mathrm{curv}^{\tilde{g}}\left( X,W+\tau
V\right) \geq \frac{1}{100}\mathrm{curv}^{\tilde{g}}\left( X,W\right) >0
\end{equation*}%
on the Gromoll--Meyer sphere. It turns out that this estimate is not good
enough, but one like inequality \ref{kappa of nu} is.
\end{remark}

\begin{remark}
By carefully considering the exponents in Proposition \ref{Compression of
psi'} one can also make more precise statements about the behavior of
allowable functions $\varkappa $ near $0.$ We will not need this, and so
have omitted it.
\end{remark}

\begin{proof}
For the moment assume that $\varphi \equiv 1,$ (i.e., the orthogonal partial
conformal change is not performed, and the resulting metric is called $\bar{g%
}.$)

Combining the previous lemma with our new hypothesis that there is a $C>0$
so that%
\begin{equation*}
\frac{g\left( \frac{W^{\mathcal{H}}}{\left\vert W^{\mathcal{H}}\right\vert }%
,A_{X}V\right) ^{2}}{\mathrm{curv}^{g}\left( X,V\right) }\leq C
\end{equation*}%
for all $\nu ,$ we conclude is that the minimum of $Q\left( \tau \right) $
satisfies%
\begin{eqnarray*}
&&\mathrm{curv}^{\bar{g}}\left( X,W\right) -\frac{R^{\bar{g}}\left(
W,X,X,V\right) ^{2}}{\mathrm{curv}^{\bar{g}}\left( X,V\right) } \\
&\geq &s^{4}\left( D_{X}\psi \right) ^{2}\left( 1-C\right) +s^{4}\frac{\psi
^{2}}{\left\vert W\right\vert ^{2}}\left( D_{X}\psi \right) ^{2}+s^{4}\frac{%
\psi ^{2}}{\left\vert W\right\vert ^{2}}D_{X}\left( \psi D_{X}\psi \right)
-s^{4}I^{\prime \prime }\left\vert W\right\vert ^{2}+O\left( s^{6}\right)
\end{eqnarray*}

It follows from Proposition \ref{integrals of psi tilde} that the sum of the
first three terms on the right hand side has a negative integral over an
integral curve $\gamma $ of $X$ that is parameterized as in Proposition \ref%
{integrals of psi tilde}. So the metric $\bar{g}$ can not satisfy our
conclusion. Depending on the precise value of $C$ we may even get that the
minimum of $Q\left( \tau \right) $ is negative somewhere along $\gamma $ for
all choices of $I^{\prime \prime }.$ In any event, our conclusion is false
without the orthogonal partial conformal change.

It follows from Theorem \ref{C^1 close copy(1)} that the orthogonal partial
conformal change does not affect $\mathrm{curv}\left( X,W\right) $ and $%
R\left( W,X,X,V\right) .$ Its effect on $\mathrm{curv}\left( X,V\right) $ is
given in Proposition \ref{redistr curv} and is 
\begin{equation}
\mathrm{curv}^{\tilde{g}}\left( X,V\right) =\mathrm{curv}^{g}\left(
X,V\right) -\varphi ^{\prime \prime }\left\vert V\right\vert
_{g}^{2}\left\vert X\right\vert _{g}^{2}+O\left( C^{1}\right) .
\label{redistr eqn}
\end{equation}%
where we use $\varphi ^{\prime \prime }$ for $D_{X}D_{X}\left( \varphi
\right) .$ The goal will now be to select $\varphi ^{\prime \prime }$
appropriately so as to adjust our estimate for 
\begin{equation*}
\left( D_{X}\psi \right) ^{2}\frac{\tilde{g}\left( \frac{W^{\mathcal{H}}}{%
\left\vert W^{\mathcal{H}}\right\vert },A_{X}V\right) ^{2}}{\mathrm{curv}^{%
\tilde{g}}\left( X,V\right) }
\end{equation*}

Recall that 
\begin{equation*}
\varphi \equiv f\circ r
\end{equation*}%
where $r$ is a smooth distance function with gradient $X$ and $f:\mathbb{R}%
\longrightarrow \mathbb{R}$, and is constant outside of a compact interval, $%
\left[ a,b\right] $. So in fact $\varphi ^{\prime \prime }=f^{\prime \prime
}\circ r.$ Since $f$ is constant outside of $\left[ a,b\right] $%
\begin{equation*}
\int_{\left[ a,b\right] }f^{\prime \prime }=0.
\end{equation*}%
Equation \ref{redistr eqn} therefore gives us a way to redistribute $\mathrm{%
curv}\left( X,V\right) $ along the integral curves of $X.$

Our curvature compression result, Proposition \ref{Compression of psi'}, and
our estimate for the minimum of $Q\left( \tau \right) $ together suggest an
appealing choice for $f^{\prime \prime }.$ Indeed Proposition \ref%
{Compression of psi'} says, for example, that $\left( D_{X}\psi \right)
^{2}\leq \nu ^{4.6}$ outside of an interval like $\left[ 0,\varkappa \left(
\nu \right) \right] $. We choose $f^{\prime \prime }$ to be negative (and
relatively large in absolute value) on an interval like $\left[ 0,O\left(
\nu \right) \right] $ and \textquotedblleft pay for this\textquotedblright\
by having $f^{\prime \prime }$ be positive (but relatively small) on $\left[
\varkappa \left( \nu \right) ,\frac{\pi }{4}\right] .$ With such a choice of 
$\varphi $ we can make the integral over any integral curve of $X$ satisfy 
\begin{equation*}
\varkappa \left( \nu \right) \int \mathrm{curv}^{\tilde{g}}\left( X,W\right)
>\int \frac{R^{\tilde{g}}\left( W,X,X,V\right) ^{2}}{\mathrm{curv}^{\tilde{g}%
}\left( X,V\right) }
\end{equation*}%
for the appropriately chosen function $\varkappa \left( \nu \right) $ with $%
\lim_{\nu \rightarrow 0}$ $\varkappa \left( \nu \right) =0.$ Indeed we have
made the denominator $\mathrm{curv}^{\tilde{g}}\left( X,V\right) $ larger on
the region $\left[ 0,O\left( \nu \right) \right] $ where $R^{\tilde{g}%
}\left( W,X,X,V\right) ^{2}$ is relatively large. We have done this at the
expense of making it very slightly smaller on $\left[ \varkappa \left( \nu
\right) ,\frac{\pi }{4}\right] ,$ but on this region $R^{\tilde{g}}\left(
W,X,X,V\right) ^{2}$ is very small. So our redistribution of $\mathrm{curv}%
\left( X,V\right) $ does in fact give us the desired inequality in an
integral sense.

We obtain the point wise inequality by combining the integral inequality
with a judicious choice of $I^{\prime \prime }.$ Namely that it be
sufficiently negative on the the complement of $\left[ 0,O\left( \nu \right) %
\right] .$
\end{proof}

\end{document}